\documentclass[11pt]{article}
\usepackage{amsmath,amsthm,amsfonts,amssymb,enumerate,calc}
\usepackage[colorlinks=true,citecolor=black,linkcolor=black,urlcolor=blue]{hyperref}
\usepackage{algorithm}
\usepackage{algorithmic} 
 \floatname{algorithm}{Algorithm}

\usepackage[margin=31mm]{geometry}
\usepackage[numbers,sort&compress]{natbib}

\renewcommand{\baselinestretch}{1.1}
\renewcommand{\thefootnote}{\fnsymbol{footnote}}	
\allowdisplaybreaks

\newcommand{\arXiv}[1]{arXiv:\,\href{http://arxiv.org/abs/#1}{#1}}
\newcommand{\msn}[1]{MR:\,\href{http://www.ams.org/mathscinet-getitem?mr=MR#1}{#1}}
\newcommand{\MSN}[2]{MR:\,\href{http://www.ams.org/mathscinet-getitem?mr=MR#1}{#1}}
\newcommand{\doi}[1]{doi:\,\href{http://dx.doi.org/#1}{#1}}

\newcommand{\coeff}{\left( 1+ \frac{1}{\Delta^{1/3}-1}+\frac{1}{\Delta^{1/3}} \right) }
\newcommand{\Dp}{D'}
\newcommand{\Dpp}{D''}
\newcommand{\SEQ}{\textsc{Seq}}

\newcommand{\DEF}[1]{\emph{#1}}

\newcommand{\N}{\mathbb{N}} 
\newcommand{\eps}{\varepsilon}
\newcommand{\mc}{\mathcal} 
\newcommand{\W}{\mc{W}} 
 
\newcommand{\RR}{\mc{R}} 
\newcommand{\DD}{\mc{D}} 
\newcommand{\F}{\mc{F}} 
\newcommand{\A}{\mc{A}} 
\newcommand{\T}{\mc{T}} 

\renewcommand{\S}{\Lambda}

\newcommand{\maxclique}{\omega}
\newcommand{\width}{\theta}

\theoremstyle{plain}
\newtheorem{theorem}{Theorem}
\newtheorem{lemma}[theorem]{Lemma}
\newtheorem{corollary}[theorem]{Corollary}

\newtheorem{claim}[theorem]{Claim}
\theoremstyle{definition}

\DeclareMathOperator{\tr}{tr}

\newcommand{\tc}{\pi_{\textup{ch}}}
\newcommand{\st}{\chi_{\textup{st}}}
\newcommand{\ceil}[1]{\lceil{#1}\rceil}
\newcommand{\floor}[1]{\lfloor{#1}\rfloor}
\newcommand{\half}{\ensuremath{\protect\tfrac{1}{2}}}

\begin{document}

\title{\bf Nonrepetitive Colouring via Entropy Compression}

\author{
Vida Dujmovi{\'c}\,\footnotemark[4] \qquad 
Gwena\"el Joret\,\footnotemark[2] \qquad 
Jakub Kozik \footnotemark[5] \qquad 
David~R.~Wood\,\footnotemark[3]}

\date{}

\maketitle

\begin{abstract}
A vertex colouring of a graph is \emph{nonrepetitive} if there is no path whose first half receives the same sequence of colours as the second half. A graph is nonrepetitively $\ell$-choosable if given lists of at least $\ell$ colours at each vertex, there is a nonrepetitive colouring such that each vertex is coloured from its own list. It is known that, for some constant $c$, every graph with maximum degree $\Delta$ is $c\Delta^2$-choosable. We prove this result with $c=1$ (ignoring lower order terms). We then prove that every subdivision of a graph with sufficiently many division vertices per edge is nonrepetitively 5-choosable. The proofs of both these results are based on the Moser-Tardos entropy-compression method, and a recent extension by Grytczuk, Kozik and Micek for the nonrepetitive choosability of paths. Finally we prove that graphs with pathwidth $\width$ are nonrepetitively $\mathcal{O}(\width^{2})$-colourable.
\end{abstract}


\footnotetext[4]{School of Mathematics and Statistics \& Department of Systems and Computer Engineering, Carleton University, Ottawa, Canada (\texttt{vdujmovic@math.carleton.ca}). Supported by NSERC and an Endeavour Fellowship from the Australian Government.}

\footnotetext[2]{Department of Mathematics and Statistics, The University of Melbourne, and School of Mathematical Sciences, Monash University, Melbourne, Australia. (\texttt{gwenael.joret@unimelb.edu.au}). 
Supported by the Belgian Fund for Scientific Research (F.R.S.--FNRS), by an Endeavour Fellowship from the Australian Government, and by a DECRA Fellowship from the Australian Research Council.}

\footnotetext[5]{Theoretical Computer Science Department, Faculty of Mathematics and Computer Science, Jagiellonian University, ul. Prof. St. Lojasiewicza 6, 30-348 Krakow, Poland (\texttt{jkozik@tcs.uj.edu.pl}).  Supported by the Polish National Science Center (DEC-2011/01/D/ST1/04412).}

\footnotetext[3]{School of Mathematical Sciences, Monash University, Melbourne, Australia  (\texttt{david.wood@monash.edu}). Research supported by  the Australian Research Council.}

\section{Introduction}

\renewcommand{\thefootnote}{\arabic{footnote}}

A colouring of a graph\footnote{We consider simple, finite, undirected
  graphs $G$ with vertex set $V(G)$, edge set $E(G)$, maximum degree
  $\Delta(G)$, and maximum clique size $\maxclique(G)$. The
  \DEF{neighbourhood} of a vertex $v\in V(G)$ is $N_G(v):=\{w\in
  V(G):vw\in E(G)\}$. The \DEF{neighbourhood} of a set $S\subseteq
  V(G)$ is $N_G(S):=\bigcup\{N_G(x):x\in S\}\setminus S$. The degree
  of a vertex $v\in V(G)$ is $\deg_G(v):=|N_G(v)|$. We use $N(v)$ and
  $N(S)$ and $\deg(v)$ if the graph $G$ is clear from the context.} is
\DEF{nonrepetitive} if there is no path $P$ such that the first half
of $P$ receives the same sequence of colours as the second half of
$P$. More precisely, a $k$-\DEF{colouring} of a graph $G$ is a
function $\psi$ that assigns one of $k$ colours to each vertex of $G$.
A path is \DEF{even} if its order is even.  An even path
$v_1,v_2,\dots,v_{2t}$ of $G$ is \DEF{repetitively} coloured by $\psi$
if $\psi(v_i)=\psi(v_{t+i})$ for all $i\in[1,t]:=\{1,2,\dots,t\}$. A
colouring $\psi$ is \DEF{nonrepetitive} if no path of $G$ is
repetitively coloured by $\psi$. The \DEF{nonrepetitive chromatic
  number} $\pi(G)$ is the minimum integer $k$ such that $G$ has a
nonrepetitive $k$-colouring.

Observe that every nonrepetitive colouring is \DEF{proper}, in the
sense that adjacent vertices receive distinct colours. Moreover, a
nonrepetitive colouring has no $2$-coloured $P_4$ (a path on four
vertices). A proper colouring with no $2$-coloured $P_4$ is called a
\DEF{star colouring} since each bichromatic subgraph is a star forest;
see \citep{Albertson-EJC04, Grunbaum73,
  FRR-JGT04,Borodin-DM79,NesOdM-03,Wood-DMTCS05}. The \DEF{star
  chromatic number} $\st(G)$ is the minimum number of colours in a
star colouring of $G$. Thus
\begin{equation*}
  \chi(G)\leq\st(G)\leq\pi(G)\enspace.
\end{equation*}

The seminal result in this field is by \citet{Thue06}, who in 1906 proved\footnote{A nonrepetitive $3$-colouring of $P_n$  is obtained as follows. Given a nonrepetitive sequence over $\{1,2,3\}$, replace each $1$ by the sequence $12312$,   replace each $2$ by the sequence $131232$, and replace each $3$ by   the sequence $1323132$. It follows that the new sequence is nonrepetitive. Thus arbitrarily long paths can be nonrepetitively $3$-coloured.} that every path is nonrepetitively 3-colourable. Nonrepetitive colourings have recently been widely studied \citep{Currie-EJC02,HJSS,PZ09,BreakingRhythm,BV-NonRepVertex07,
  BaratWood-EJC08,BK-AC04,GKM11,Grytczuk-EJC02,CG-ENDM07,BGKNP-NonRepTree-DM07, Grytczuk,Currie-TCS05,Manin,NOW,HJ-DM11,KP-DM08,FOOZ,Pegden11,BV-NonRepEdge08, AGHR-RSA02,GPZ,JS09,DS09,Kolipaka,BC13,Keszegh,KozikMicek,Grytczuk,Grytczuk-DM08,Gryczuk-IJMMS07,CSZ}; see the surveys \citep{Grytczuk,Grytczuk-DM08,Gryczuk-IJMMS07,CSZ}.

The contributions of this paper concern three different
generalisations of the result of Thue: bounded degree graphs, graph
subdivisions, and graphs of bounded pathwidth.

\subsection{Bounded Degree}

In a sweeping generalisation of Thue's result, \citet{AGHR-RSA02}
proved\footnote{The theorem of \citet{AGHR-RSA02} was actually for
  edge colourings, but it is easily seen that the method works in the
  more general setting of vertex colourings.} that for some constant
$c$ and for every graph $G$ with maximum degree $\Delta\geq 1$,
\begin{equation}
  \label{BoundedDegree}
  \pi(G)\leq c\,\Delta^2\enspace.
\end{equation}
Moreover, the bound in \eqref{BoundedDegree} is almost
tight---\citet{AGHR-RSA02} proved that there are graphs with maximum
degree $\Delta$ that are not nonrepetitively $(c\Delta^2/\log\Delta)$-colourable for some constant $c$.

The bound in \eqref{BoundedDegree}, in fact, holds in the stronger
setting of nonrepetitive list colourings. A \DEF{list assignment} of a
graph $G$ is a function $L$ that assigns a set $L(v)$ of colours to
each vertex $v\in V(G)$. Then $G$ is \DEF{nonrepetitively
  $L$-colourable} if there is a nonrepetitive colouring of $G$, such
that each vertex $v\in V(G)$ is assigned a colour in $L(v)$. And $G$
is \DEF{nonrepetitively $\ell$-choosable} if for every list assignment
$L$ of $G$ such that $|L(v)|\geq \ell$ for each vertex $v\in V(G)$, there
is a nonrepetitive $L$-colouring of $G$. The \DEF{nonrepetitive choice
  number} $\tc(G)$ is the minimum integer $\ell$ such that $G$ is
nonrepetitively $\ell$-choosable. Clearly, $\pi(G)\leq\tc(G)$.

All the known proofs of \eqref{BoundedDegree} are based on the
Lov{\'a}sz Local Lemma, and thus are easily seen to prove the stronger
result that
\begin{equation}
  \label{BoundedDegreeChoose}
  \tc(G)\leq c\,\Delta(G)^2\enspace.
\end{equation}
\citet{AGHR-RSA02} originally proved \eqref{BoundedDegreeChoose} with
$c=2 e^{16}$, which was improved to $36$ by \citet{Grytczuk} and then
to $16$ again by \citet{Gryczuk-IJMMS07}. Prior to the present paper, the best bound was $\tc(G)\leq 12.92(\Delta(G)-1)^2$ by
\citet{HJ-DM11} (assuming $\Delta(G)\geq2$). We improve the constant $c$ to $1$. 

\begin{theorem}
  \label{thm:One} For every graph $G$ with maximum degree $\Delta$,
$$\tc(G)\leq (1+o(1))\Delta^2\enspace.$$
\end{theorem}

The proof of Theorem~\ref{thm:One} is based on the celebrated
entropy-compression method of \citet{MoserTardos}, and more precisely
on an extension by \citet{GKM11} for nonrepetitive sequences (or
equivalently, nonrepetitive colourings of paths).  The latter authors
considered the following variant of the Moser-Tardos algorithm for
nonrepetitively colouring paths. Start at the first vertex of the path
and repeat the following step until a valid colouring is produced:
Randomly colour the current vertex. If doing so creates a repetitively
coloured subpath $P$, then uncolour the second half of $P$ and let the
new current vertex be the first uncoloured vertex on the path.
Otherwise, go to the next vertex in the path.  \citet{GKM11} used this
algorithm to obtain a short proof that paths are nonrepetitively
$4$-choosable, which was first proved by \citet{GPZ} using the
Lov\'asz Local Lemma. (It is open whether every path is
nonrepetitively $3$-choosable.) Our proof of Theorem~\ref{thm:One}
generalises this method for graphs of bounded degree.  
While the main conclusion of the Moser-Tardos method
was a constructive proof of the Lov{\'a}sz Local Lemma, as
\citet{KS-STOC} write, ``variants of the Moser-Tardos algorithm can be
useful in existence proofs''. Our result is further evidence of this
claim.

Note that the methods developed in the proof of Theorem~\ref{thm:One} have subsequently 
been applied to other graph colouring problems  \cite{Przybyo,EsperetParreau,Przybyo13} and also 
in pattern avoidance \cite{OchemPinlou}.

\subsection{Subdivisions}

A \DEF{subdivision} of a graph $G$ is a graph obtained from $G$ by
replacing the edges of $G$ with internally disjoint paths, where the
path replacing $vw$ has endpoints $v$ and $w$. In a beautiful
generalisation of Thue's theorem, \citet{PZ09} proved that every graph
has a nonrepetitively 3-colourable subdivision (improving on analogous
5- and 4-colour results by \citet{Gryczuk-IJMMS07} and
\citet{BaratWood-EJC08} respectively). For each of these theorems, the
number of division vertices per edge is $\mathcal{O}(n)$ or
$\mathcal{O}(n^2)$ for $n$-vertex graphs. Improving these bounds,
\citet{NOW} proved that every graph has a nonrepetitively
17-colourable subdivision with $\mathcal{O}(\log n)$ division vertices
per edge, and that $\Omega(\log n)$ division vertices are needed on
some edge of a nonrepetitively $\mathcal{O}(1)$-colourable subdivision
of $K_n$.  Here we prove that every graph has a nonrepetitively
$\mathcal{O}(1)$-\DEF{choosable} subdivision, which solves an open
problem by \citet{GPZ}. All logarithms are binary.

\begin{theorem}
  \label{thm:Subdivision}
  Let $G$ be a subdivision of a graph $H$, such that each edge $vw\in
  E(H)$ is subdivided at least
  $\ceil{10^{5}\log(\deg(v)+1)}+\ceil{10^{5}\log(\deg(w)+1)}+2$ times in
  $G$.  Then $$\tc(G)\leq 5\enspace.$$
\end{theorem}

Theorem~\ref{thm:Subdivision} is stronger than the above subdivision
results in the following respects: (1) it is for choosability not just
colourability; (2) it applies to every subdivision with \emph{at
  least} a certain number of division vertices per edge, and (3) the
required number of division vertices per edge is asymptotically fewer
than for the above results. Of course, Theorem~\ref{thm:Subdivision}
is weaker than the results in \cite{PZ09, BaratWood-EJC08} mentioned above 
in that the number of colours is 5.

Theorem~\ref{thm:Subdivision} is also proved using the
entropy-compression method mentioned above. An analogous theorem with
more colours and $\mathcal{O}(\log\Delta(G))$ division vertices per
edges can be proved using the Lov{\'a}sz Local Lemma; see Appendix~\ref{subLLL}.

\subsection{Pathwidth}

Thue's result was generalised in a different direction by
\citet{BGKNP-NonRepTree-DM07}, who proved that every tree is
nonrepetitively 4-colourable\footnote{No such result is possible for
choosability---\citet{FOOZ} proved that trees have arbitrarily high
nonrepetitive choice number. On the other hand, \citet{KozikMicek} proved
that $\tc(T)\leq \mathcal{O}(\Delta^{1+\eps})$ for every tree $T$ with maximum
degree $\Delta$.}. This result was further generalised by
considering treewidth\footnote{A \DEF{tree decomposition} of a graph
  $G$ consists of a tree $T$ and a set $\{B_x\subseteq V(G):x\in
  V(T)\}$ of subsets of vertices of $G$, called \DEF{bags}, indexed by
  the vertices of $T$, such that (1) the endpoints of each edge of $G$
  appear in some bag, and (2) for each vertex $v$ of $G$ the set
  $\{x\in V(T):v\in B_x\}$ is nonempty and induces a connected subtree
  of $T$. The \DEF{width} of the tree decomposition is
  $\max\{|B_x|-1:x\in V(T)\}$. The \DEF{treewidth} of $G$ is the
  minimum width of a tree decomposition of $G$. A \DEF{path
    decomposition} is a tree decomposition whose underlying tree is a
  path. Thus a path decomposition is simply defined by a sequence of
  bags $B_1,\dots,B_p$. The \DEF{pathwidth} of $G$ is the minimum
  width of a path decomposition of $G$.}, which is a parameter that
measures how similar a graph is to a tree. \citet{KP-DM08} and
\citet{BV-NonRepVertex07} independently proved that graphs of bounded
treewidth have bounded nonrepetitive chromatic number. The best upper
bound is due to \citet{KP-DM08}, who proved that $\pi(G)\leq 4^\width$ for
every graph $G$ with treewidth $\width$.  The best lower bound is due to
\citet{Albertson-EJC04}, who described graphs $G$ with treewidth $\width$
and $\pi(G)\geq\st(G)=\binom{\width+2}{2}$.  Thus there is a quadratic
lower bound on $\pi$ in terms of treewidth. It is open whether $\pi$
is bounded from above by a polynomial function of treewidth. We prove
the following related result.

\begin{theorem}
  \label{thm:Pathwidth}
  For every graph $G$ with pathwidth $\width$,
$$\pi(G)\leq 2\width^2+6\width+1\enspace.$$
\end{theorem}

It is open whether $\pi(G)\in\mathcal{O}(\width)$ for every graph $G$ with
pathwidth $\width$. For treewidth, a quadratic lower bound on $\pi$ follows
from the quadratic lower bound on $\st$, as explained above. However,
we show that no such result holds for pathwidth.

\begin{theorem}
  \label{thm:StarPathwidth}
For every graph $G$ with pathwidth $\width$,
$$\st(G)\leq 3\width+1\enspace.$$
\end{theorem}

\section{An Algorithm}

This section presents and analyses an algorithm for nonrepetitively list colouring a graph. This machinery will be used to prove Theorems~\ref{thm:One} and \ref{thm:Subdivision} in the following sections.

If a set $X$ is linearly ordered according to some fixed ordering and $e \in X$, then the \DEF{index} of $e$ in $X$ is the position of $e$ in this ordering of $X$. Such an ordering induces in a natural way an ordering of each subset $Y$ of $X$, so that the index of an element $e \in Y$ {\em in $Y$} is well defined.

Let $G$ be a fixed $n$-vertex graph. Assume that $V(G)$ is ordered
according to some arbitrary linear ordering.  Let $L$ be a list
assignment for $G$.  Assume each list in $L$ has size $\ell$.
Identify colours with nonnegative integers. Thus the colours in $L(v)$
are ordered in a natural way, for each $v\in V(G)$.  Without loss of
generality, the colour $0$ is in none of these lists.  In what
follows, we consider an uncoloured vertex to be coloured $0$.  A
\DEF{precolouring} of $G$ is a colouring $\psi$ of $G$ such that
$\psi(v) \in L(v)\cup\{0\}$ for each $v\in V(G)$.  If $\psi(v)\neq 0$
then $v$ is said to be \DEF{precoloured} by $\psi$.

For each path $P$ of $G$ with $2k$ vertices, for each subset
$X\subseteq V(G) - V(P)$, and for each vertex $v \in V(P)$, define
$\lambda(P, X, v)$ to be the sequence $(\lambda_{1}, \dots, \lambda_{2k})$ obtained as
follows: Let $x, y$ be the two endpoints of $P$, with $v$ closer to
$y$ than to $x$ in $P$.  Let $v_{1}, \dots, v_{p}$ be the vertices of
$P$ from $v_{1}:=v$ to $v_{p}:=x$ defined by $P$, in order. (Observe
that $p \geq 2$ since $v \neq x$.)  Let $\lambda_{1}$ be the index of
$v_{2}$ in $N(v_{1}) - X$, and for each $i\in [2, p-1]$, let $\lambda_{i}$
be the index of $v_{i+1}$ in $N(v_{i}) - (X \cup \{v_{i-1}\})$.  Let
$\lambda_{p} := -1$.  If $v = y$ then $p=2k$ and the sequence $(\lambda_{1},
\dots, \lambda_{2k})$ is completely defined. If $v \neq y$ then let $q:=
2k-p + 1$ and let $w_{1}, \dots, w_{q}$ be the vertices of $P$ from
$w_{1}:=v$ to $w_{q}:=y$ defined by $P$, in order.  Let $\lambda_{p+1}$ be
the index of $w_{2}$ in $N(w_{1}) - (X \cup \{v_{2}\})$, and for each
$i\in [2, q-1]$, let $\lambda_{p + i}$ be the index of $w_{i+1}$ in
$N(w_{i}) - (X \cup \{w_{i-1}\})$.

An important feature of the above encoding of the triple $(P, X, v)$
as a sequence $\lambda(P, X, v)$ is that it can be reversed, as we now
explain.

\begin{lemma}
  \label{lem:P}
  Suppose $\lambda = \lambda(P, X, v)$ for some even path $P$ of $G$ such that $X
  \subseteq V(G) - V(P)$, and $v\in V(P)$.  Then, given $\lambda$, $X$, and
  $v$, one can uniquely determine the path $P$.
\end{lemma}

\begin{proof}
  Let $\lambda = (\lambda_{1}, \dots, \lambda_{2k})$ and let $p\in [2, 2k]$ be the
  unique index such that $\lambda_{p} = -1$.  Let $u_{p} := v$, let
  $u_{p-1}$ be the $\lambda_{1}$-th vertex in $N(u_{p}) - X$, and for $i =
  p-2, \dots, 1$, let $u_{i}$ be the $\lambda_{p - i}$-th vertex in
  $N(u_{i+1}) - (X \cup \{u_{i+2}\})$.  Next, for $j = p+1, \dots,
  2k$, let $u_{j}$ be the $\lambda_{j}$-th vertex in $N(u_{j-1}) -
  (X\cup\{u_{j-2}\})$.  Then the vertices $u_{1},u_{2},\dots,u_{2k}$,
  in this order, determine a path $P$ of $G$ such that $\lambda(P, X,v) =
  \lambda$.

  Observe that if $P'$ is an even path of $G$ such that $X\subseteq
  V(G) - V(P')$, $v\in V(P')$, and $P'$ is distinct from $P$, then
  $\lambda(P', X, v) \neq \lambda(P, X, v)$.  Therefore, the path $P$ above is
  uniquely determined.
\end{proof}

Let $\S$ be the set of all sequences $\lambda(P, X, v)$ where $P$ is an even
path in $G$, $X\subseteq V(G) - V(P)$, and $v$ is a vertex of $P$.  A
\DEF{record} is a mapping $R:\N \to \S \cup \{\varnothing\}$.  The
\DEF{empty record} is the record $R$ such that $R(i) = \varnothing$
for all $i \in \N$.

A \DEF{priority function} is a function $f$ that associates to each
nonempty subset $X$ of $V(G)$ a vertex $f(X) \in X$.  Consider
Algorithm~\ref{alg:delta}, which (for a fixed graph $G$, a list
assignment $L$, a priority function $f$, and a precolouring $\psi$)
takes as input a positive integer $t$ and a vector $(c_{1},\dots,c_{t}) \in [1,\ell]^{t}$. Note that precoloured vertices and a
specific priority function will only be needed when proving the result
on subdivisions.  We thus invite the reader to first consider the set
$Q$ of precoloured vertices to be empty, and the priority function $f$
to be arbitrary (for instance, $f(X)$ could be the first vertex in $X$
in the fixed ordering of $V(G)$). Also note that the choice of the
repetitively coloured path $P$ in the algorithm is assumed
to be consistent; that is, according to some (arbitrary) fixed
deterministic rule.

\begin{algorithm}[h]
  \algsetup{indent=2em}
  \caption{$L$-colouring the graph $G$, where $f$ is a priority
    function, $\psi$ is a precolouring of $G$, and $Q$ is the set of
    precoloured vertices under $\psi$.}
  \label{alg:delta}
  \begin{algorithmic}
    \REQUIRE $(c_{1}, \dots, c_{t}) \in [1,\ell]^{t}$ \ENSURE a
    (possibly invalid) colouring $\phi$ and a record $R$ \STATE $i
    \gets 1$ \STATE $\phi \gets \psi$ \STATE $R \gets $ empty record
    \STATE $X \gets V(G) - Q$ \WHILE {$i\leq t$ and $X \neq
      \varnothing$} \STATE $v \gets f(X)$ \STATE $\phi(v) \gets $
    $c_{i}$-th colour in $L(v)$ 
    \STATE $X \gets X -  \{v\}$
    \IF {$G$ contains a repetitively
      coloured path $P$} \STATE divide $P$ into first half $P_{1}$ and
    second half $P_{2}$ so that $v\in V(P_{2})$ \FOR {$w\in V(P_{2}) -
      Q$} \STATE $\phi(w) \gets 0$
    \ENDFOR
    \STATE $R(i) \gets \lambda(P, X, v)$ \STATE $X \gets X \cup (V(P_{2}) -
    Q)$ \ELSE \STATE $R(i) \gets \varnothing$ 
    \ENDIF
    \STATE $i\gets i+1$
    \ENDWHILE
    \RETURN return $\phi$, $R$
  \end{algorithmic}
\end{algorithm}

Say that the algorithm \DEF{succeeds} if it terminates
with $X = \varnothing$, and \DEF{fails} otherwise.  It is easily seen
that if the algorithm succeeds, then the produced colouring $\phi$ is
a nonrepetitive $L$-colouring of $G$.  For a given integer $t \geq 1$,
let $\F_{t}$ be the set of vectors $(c_{1},\dots,c_{t})\in[1,\ell]^{t}$ on which the algorithm fails.  Let $\A_{t}$ be the set of pairs $(\phi, R)$ that are produced by the algorithm on vectors in $\F_{t}$.  Let $\RR_{t}$ be the set of distinct records $R$ that can be produced by the algorithm
on vectors in $\F_{t}$.  For $R \in \RR_{t}$, let $\F_{t, R}$ be the
set of vectors $(c_{1}, \dots, c_{t}) \in \F_{t}$ on which the
algorithm produces record $R$.  (Thus $\F_{t, R} \neq \varnothing$.)
For a vector $(c_{1}, \dots, c_{t}) \in \F_{t}$, let the \DEF{trace}
$\tr(c_{1}, \dots, c_{t})$ be the vector $(X_{1}, \dots, X_{t})$ where
$X_{i}$ is the set $X$ at the beginning of the $i$-th iteration of the
while-loop of the algorithm on input $(c_{1}, \dots, c_{t})$, for each
$i\in [1,t]$. (Observe that $X_1$ always equals $V(G) - Q$.)  Finally,
let $\T_{t} := \{\tr(c_{1}, \dots, c_{t}) : (c_{1}, \dots, c_{t}) \in
\F_{t}\}$.

The next lemma shows that for a fixed record $R \in \RR_{t}$, all the
vectors in $\F_{t, R}$ have the same trace.

\begin{lemma}
  \label{lem:trace}
  For every $t\geq 1$ there exists a function $h_{t}: \RR_{t} \to
  \T_{t}$ such that for each $R \in \RR_{t}$ and each $(c_{1}, \dots,
  c_{t}) \in \F_{t, R}$ we have $\tr(c_{1}, \dots, c_{t}) = h_{t}(R)$.
\end{lemma}

\begin{proof}
  We construct $h_{t}$ by induction on $t$.  For $t=1$ simply let
  $h_{t}(R):= ( V(G) - Q )$ for each $R\in \RR_{t}$.

  Now assume that $t \geq 2$. Let $R \in \RR_{t}$.  Let $R'$ be the
  record obtained from $R$ by setting $R'(i) := R(i)$ for each $i \in
  \N$ such that $i \neq t$, and $R'(t) := \varnothing$.  Then $R' \in
  \RR_{t-1}$, and by induction, $h_{t-1}(R') = (X_{1}, \dots,
  X_{t-1})$ for some $(X_{1}, \dots, X_{t-1}) \in \T_{t-1}$.  Let
  $v_{t-1} := f(X_{t-1})$.

  First suppose that $R(t-1) = \varnothing$. Let $X_{t} := X_{t-1} -
  \{v_{t-1}\}$ and $h_{t}(R) := (X_{1}, \dots, X_{t-1}, X_{t})$.
  Consider a vector $(c_{1}, \dots, c_{t}) \in \F_{t, R}$.  Then
  $(c_{1}, \dots, c_{t-1}) \in \F_{t-1, R'}$, and by induction
  $\tr(c_{1}, \dots, c_{t-1}) = (X_{1}, \dots, X_{t-1})$.  Thus at the
  beginning of the $(t-1)$-th iteration of the while-loop in the
  algorithm on input $(c_{1}, \dots, c_{t})$, the current record is
  $R'$, and $v = v_{t-1}$ and $X = X_{t-1}$. Since $R(t-1) =
  \varnothing$, the algorithm subsequently coloured $v_{t-1}$ without
  creating any repetitively coloured path, implying that $X = X_{t-1}
  - \{v_{t-1}\} = X_{t}$ at the beginning of the $t$-th iteration.
  Hence $\tr(c_{1}, \dots, c_{t}) = (X_{1}, \dots, X_{t-1}, X_{t}) =
  h_{t}(R)$, as desired.

  Now assume that $R(t-1) = \lambda$ for some $\lambda \in \S$.  Using
  Lemma~\ref{lem:P}, let $P$ be the path of $G$ determined by $\lambda$,
  $X_{t-1}$, and $v_{t-1}$.  Let $P_{1}$ and $P_{2}$ denote the two
  halves of $P$, so that $v_{t-1} \in V(P_{2})$.  Let $X_{t} :=
  X_{t-1} \cup (V(P_{2}) - Q)$ and $h_{t}(R) := (X_{1}, \dots,
  X_{t-1}, X_{t})$.  Consider a vector $(c_{1}, \dots, c_{t}) \in
  \F_{t, R}$.  Then $(c_{1}, \dots, c_{t-1}) \in \F_{t-1, R'}$, and by
  induction $\tr(c_{1}, \dots, c_{t-1}) = (X_{1}, \dots, X_{t-1})$.
  As before, at the beginning of the $(t-1)$-th iteration of the
  while-loop in the algorithm on input $(c_{1}, \dots, c_{t})$, the
  current record is $R'$, and $v = v_{t-1}$ and $X = X_{t-1}$.  Then,
  after colouring $v$, the path $P$ is repetitively coloured, and all
  vertices in $P_{2}$ are subsequently uncoloured, except for those in
  $Q$. Hence we have $X = X_{t-1} \cup (V(P_{2}) - Q) = X_{t}$ at the
  beginning of the $t$-th iteration.  Therefore $\tr(c_{1}, \dots,
  c_{t}) = (X_{1}, \dots, X_{t-1}, X_{t}) = h_{t}(R)$.
\end{proof}

\begin{lemma}
  \label{lem:lossless}
  For every $(\phi, R) \in \A_{t}$ there is a unique vector $(c_{1},
  \dots, c_{t}) \in \F_{t}$ such that the algorithm
  produces $(\phi, R)$ on input $(c_{1}, \dots, c_{t})$.
\end{lemma}

\begin{proof}
  The proof is by induction on $t$.  This claim is true for $t=1$
  since in that case the unique vector $(c_{1}) \in \F_{1}$ yielding
  $(\phi, R)$ is the one where $c_{1}$ is the index of colour
  $\phi(v_{1})$ in the list $L(v_{1})$, where $v_{1} := f(V(G) - Q)$.

  Now assume that $t\geq 2$.  Let $(X_{1}, \dots, X_{t}) := h_{t}(R)$,
  where $h_{t}$ is the function in Lemma~\ref{lem:trace}.  Let $v_{t}
  := f(X_{t})$. (Recall that $X_{t} \neq \varnothing$.)  Let $R'$ be
  the record obtained from $R$ by setting $R'(i) := R(i)$ for each $i
  \in \N$ such that $i \neq t$, and $R'(t) := \varnothing$.  Then $R'
  \in \RR_{t-1}$, and $h_{t-1}(R') = (X_{1}, \dots, X_{t-1})$.

  First suppose that $R(t)=\varnothing$.  Let $\phi'$ be the colouring
  obtained from $\phi$ by setting $\phi'(v_{t}) := 0$ and $\phi'(w) :=
  \phi(w)$ for each $w \in V(G) - \{v_{t}\}$.  Then $(\phi', R') \in
  \A_{t-1}$, and by induction there is a unique input vector
  $(c'_{1},\dots,c'_{t-1}) \in \F_{t-1}$ for which the algorithm
  produces $(\phi', R')$. It follows that every vector
  $(c_{1},\dots,c_{t}) \in \F_{t}$ resulting in the pair $(\phi, R)$
  satisfies $c_{i} = c'_{i}$ for each $i \in [1, t-1]$. But  
  $c_{t}$ is also uniquely determined, since it is the index of colour
  $\phi(v_{t})$ in the list $L(v_{t})$. Hence there is a unique such
  vector $(c_{1}, \dots, c_{t})$.

  Now assume that $R(t) = \lambda$ for some $\lambda \in \S$.  Using
  Lemma~\ref{lem:P}, let $P$ be the path of $G$ determined by $\lambda$,
  $X_{t}$, and $v_{t}$.  Let $P_{1}$ and $P_{2}$ denote the two halves
  of $P$, so that $v_{t} \in V(P_{2})$.  Let $w_{1}, \dots, w_{2k}$
  denote the vertices of $P$, in order, so that $V(P_{1}) = \{w_{1},
  \dots, w_{k}\}$ and $V(P_{2}) = \{w_{k+1}, \dots, w_{2k}\}$.  Let $j
  \in [1,k]$ be the index such that $w_{k+j} = v_{t}$.  Let $\phi'$ be
  the colouring obtained from $\phi$ by setting $\phi'(v_{t}) := 0$,
  $\phi'(w_{k+i}) := \phi(w_{i})$ for each $i\in [1,k]$ such that
  $i\neq j$ and $w_{k+i} \notin Q$, and $\phi'(w) := \phi(w)$ for each
  $w \in V(G) - (V(P_{2}) - Q)$.  Then $(\phi', R') \in \A_{t-1}$, and
  by induction there is a unique vector $(c'_{1}, \dots, c'_{t-1}) \in
  \F_{t-1}$ on the input of which the algorithm produces $(\phi',
  R')$. It follows that every vector $(c_{1}, \dots, c_{t}) \in
  \F_{t}$ resulting in the pair $(\phi, R)$ satisfies $c_{i} = c'_{i}$
  for each $i \in [1, t-1]$. Moreover, $c_{t}$ is the index of colour
  $\phi(v_{t})$ in the list $L(v_{t})$, and therefore is also uniquely
  determined.
\end{proof}

Lemma~\ref{lem:lossless} implies that $|\A_{t}| = |\F_{t}|$ for all $t
\geq 1$.

Recall that $\S$ is the set of all sequences $\lambda(P, X, v)$ where $P$ is
an even path in $G$, $X\subseteq V(G) - V(P)$, and $v \in
V(P)$.  Once we fix a precolouring $\psi$ of $G$ and a priority
function $f$, as we did above, some triples $(P, X, v)$ will never be
considered by the algorithm on any input.  (For instance, this is the
case if $X$ contains a precoloured vertex.)  This leads to the following definition. 
A sequence  $\lambda \in \S$ is \DEF{realisable} (with respect to $\psi$ and $f$) if $R(i) = \lambda$ for some $t \geq 1$, $R \in \RR_{t}$, and $i \in [1,t]$. For each $k \in [1, \lfloor \tfrac{n}{2} \rfloor]$, let $\alpha_{k}$ be the number of realisable sequences of length $2k$ in $\S$. Define
$$\beta:=\max\{1,\max\{(\alpha_{k})^{1/k}: 1 \leq k \leq \floor{\tfrac{n}{2}}\}\}\enspace.$$ Thus $\beta\geq1$ and $\alpha_{k} \leq
\beta^{k}$ for each $k \in [1, \lfloor \tfrac{n}{2} \rfloor]$.
 
A \emph{substring} of some sequence or word is a subsequence of consecutive elements. A \emph{prefix} of a sequence is a substring starting  at the first element. A \DEF{Dyck word} of length $2t$ is a binary sequence with $t$ zeroes and $t$ ones such that the number of zeroes is at least the number of ones in every prefix of the sequence.  
 
Let $R\in \RR_{t}$. That is, $R$ is a record that can be produced by the algorithm on some vector $(c_{1}, \dots, c_{t}) \in[1,\ell]^{t}$ on which the algorithm fails. By Lemma~\ref{lem:trace},  $\tr(c_{1}, \dots, c_{t}) = h_{t}(R)$. That is, the vector $(X_{1}, \dots, X_{t})$ is determined by $R$, where $X_{i}$ is the set $X$ at the beginning of the $i$-th iteration of the while-loop of the algorithm on input $(c_{1}, \dots, c_{t})$. For each $i \in [1,t]$, let $r_{i}:=|X_{i+1}|-|X_i|+1$. Note that $r_i$ is the number of vertices that are uncoloured at step $i$. In particular, at step $i$, if $G$ contained no repetitively coloured path, then $r_i=0$,  and if $G$ contained a repetitively coloured path $P$ with second half $P_2$, then $r_i=|V(P_2)-Q|$. We emphasise, however, that $r_i$ is determined by $R$. Define $z(R):= t-\sum_{i=1}^{t} r_{i}$.  Observe that $z(R)$ equals the number of coloured vertices in $V(G)-Q$ at the end of any execution of the algorithm that produces the record $R$. (Recall that a vertex of colour $0$ is interpreted as being uncoloured.)  In particular,
  $z(R) \geq 1$, since there is always at least one coloured vertex, and $z(R) \leq n$. Associate with $R$ the word
$$D(R):=01^{r_{1}}01^{r_{2}}\dots01^{r_{t}}1^{z(R)} .$$  
Then $D(R)$ is a Dyck word of length $2t$.  A $0$  in $D(R)$ corresponds to the colouring of a vertex in the algorithm,  while  a $1$ corresponds to the uncolouring of a vertex, {\em except} for the last $z(R)$ $1$'s, which are added in order to ensure that the number of $0$'s and $1$'s in $D(R)$ is the same. 

  Conversely, a Dyck word $d$ is \DEF{realisable} if there exist $t \geq 1$ and $R \in \RR_{t}$ such that $D(R) = d$.    The set of realisable Dyck words of length $2t$ is denoted $\DD_{t}$.

\section{Bounded Degree Proof}

The next result is a precise version of Theorem~\ref{thm:One}.  The proof makes use of the symbolic  approach to combinatorial enumeration via generating functions. We refer the reader to the book by Flajolet and Sedgewick~\cite{FS09} for background on this topic,  as well as for undefined terms and notations. We postpone these technicalities until the end of the section. Note that we do not attempt to optimize the lower order terms in the proof of Theorem~\ref{thm:first_theorem_made_precise}. 

\begin{theorem}
\label{thm:first_theorem_made_precise}
For every graph $G$ with maximum degree $\Delta >1$,
$$\tc(G) \leq \left\lceil \coeff \Delta^2 \right\rceil\enspace.$$
\end{theorem} 
\begin{proof}
  Let $G$ be a graph with maximum degree $\Delta$. Fix an ordering of
  $V(G)$. Let $n:=|V(G)|$ and let $L$ be a list assignment of
  $G$. Assume each list in $L$ has size 
  $\ell:= \left\lceil d  \Delta ^2 \right\rceil$, where $$d:=1+ \frac{1}{\Delta^{1/3}-1}+\frac{1}{\Delta^{1/3}}\enspace.$$
Let $f$ be an arbitrary priority   function. Consider the algorithm on $G$, where none of   the vertices of $G$ are precoloured (thus $Q=\emptyset$). 	  We will prove that   $|\mathcal{A}_t |= o(d^t     \Delta^{2t})$.      It suffices to show that     $|\mathcal{R}_t| =   o(d^t \Delta^{2t})$     since the number of distinct   colourings that can be produced by the algorithm is at most
  $(\ell+1)^n$ (taking into account the extra colour 0).
	
  Let $\lambda=(\lambda_1, \ldots, \lambda_{2k})$ be a realisable sequence in $\S$. Observe
  that $\lambda_j\neq \Delta$ for each $j\in [2,2k]$. Also, there is a
  unique index $p\in [k+1,2k]$ such that $\lambda_p=-1$. Thus $\lambda_1\in
  [1,\Delta]$ and $\lambda_p=-1$ and $\lambda_j\in [1,\Delta-1]$ for each $j\in
  [1,2k]\setminus \{1,p\}$. Hence there are at most $k \Delta
  (\Delta-1)^{2(k-1)}$ realisable sequences of length $2k$ in $\S$. Therefore
  $\alpha_k < k \Delta^{2k-1}$.
	
Let $d=(d_1, \ldots,d_{2t})$ be a realisable Dyck word of length   $2t$. Suppose that $d$ has the form   $0^{l_1}1^{k_1}0^{l_2}1^{k_2}\dots 0^{l_q}1^{k_q}1$, for some  positive integers $q,l_1, \ldots, l_q, k_1, \ldots, k_q$. Note that  $\sum_{j=1}^q k_j = t-1$. Associate with the word $d$ a weight   $w(d):= k_1 k_2 \dots k_q \Delta^{-q}.$ Clearly, for every $i\in[0,k_q]$ the number of distinct records $R\in \mathcal{R}_t$ with
  $z(R)=i+1$ such that $D(R)=d$ does not exceed 
  \begin{align*}
  \alpha_{k_1} \cdots \alpha_{k_{q-1}} \alpha_{k_q-i} 
  & < k_{1}\Delta^{2k_{1} - 1}\cdots 
    k_{q-1}\Delta^{2k_{q-1} - 1}(k_{q}-i)\Delta^{2(k_{q}-i) - 1}  \\
  &\leq k_{1}\Delta^{2k_{1} - 1}\cdots k_{q}\Delta^{2k_{q} - 1}  \\
  &\leq w(d) \Delta^{2t}\enspace.
  \end{align*}
  Therefore
$$|\mathcal{R}_t| < n \cdot \Delta^{2t} \cdot \sum_{d \in \DD_{t} }w(d)\enspace.$$

  \begin{claim}
    \label{TheClaim}
$$    \sum_{d \in \DD_{t} } w(d) = o(d^{t})\enspace.$$
  \end{claim}
  \begin{proof}
    Let $\Dp$ be the set of words on the alphabet $\{0,1,2\}$ that
    \begin{itemize}
    \item do not contain substrings 21 and 02,
    \item in every nonempty prefix the number of nonzero elements 
    is strictly less than the number of zeroes, and
    \item the number of ones and twos in the whole word is one less
      than the number of zeroes.
    \end{itemize}
    Let $\gamma$ be the function that, given a word in $\Dp$, replaces
    each 2 by 1 and appends 1. Observe that the image of $\gamma$ is
    a Dyck word. Let $d$ be a realisable Dyck word of length
    $2t$. Then for every proper, nonempty prefix of $d$, the number of ones
    is strictly less than the number of zeroes. In particular, $d$
    belongs to the image of $\gamma$. We are interested in the size of
    the preimage of $d$. Suppose that $d$ has the form
    $0^{l_1}1^{k_1}0^{l_2}1^{k_2}\dots0^{l_q} 1^{k_q}1$, for some
    positive integers $q,l_1, \ldots, l_q, k_1, \ldots, k_q$. By the definition
    of $\gamma$, the elements of the preimage of $d$ are exactly the
    words of the form $0^{l_1}1^{a_1}2^{b_1}0^{l_2}1^{a_2}2^{b_2}
    \dots 0^{l_q}1^{a_q}2^{b_q}$ where for every $i\in [1,q]$ we have
    $a_i+b_i=k_i$ and $a_i>0$ and $b_i \geq 0$. Hence the size of this
    preimage is $k_1 k_2 \ldots k_q$, which equals $w(d)
    \Delta^q$. Moreover, every element of the preimage of $d$ has $t$
    zeroes and exactly $q$ substrings $01$. Let $F_{t,q}$ be the number
    of words from $\Dp$ with exactly $t$ zeroes and exactly $q$
    substrings $01$. It follows from the above observations that
$$\sum_{d \in \DD_{t} } w(d) \leq \sum_{q=0}^\infty F_{t,q}
\Delta^{-q}\enspace.$$ Define the formal power series
$$F(z,y):= \sum_{t=0}^\infty \sum_{q=0}^\infty F_{t,q} z^t y^q
\enspace.$$ 
Then
$$B(z):= F(z ,\Delta^{-1})= \sum_{t=0}^\infty z^t
\left(\sum_{q=0}^\infty F_{t,q} \Delta^{-q} \right) \enspace.$$ 
Recall that $[z^t] B(z)$ is the coefficient of $z^t$ in $B(z)$. Hence
$$\sum_{d \in \DD_{t} } w(d) \leq [z^t] B(z) \enspace.$$

We now derive a functional equation defining $F(z,y)$ by decomposing
elements of $\Dp$ recursively along the last sequence of nonzero
letters.  For $(d_1, \ldots, d_{2t-1})\in \Dp$ we say that
\emph{position $j$ visits level $k$} if the number of zeroes in $(d_1,
\ldots, d_j)$ exceeds the number of nonzero symbols by $k$.  
The sequence $(0)$ is the unique sequence 
in $\Dp$ that contains only one zero.  Consider some other
sequence $d=(d_1, \ldots, d_{2t-1})$ from $\Dp$ with $t$ zeroes that
ends with exactly $p>0$ nonzero symbols. Define $\delta_1$ to be the substring of $d$ 
starting at $d_1$ and ending at the last position that visits level $1$. 
For $i=2,3,\dots,p$, define  $\delta_i$ to be the substring of $d$
starting at the position immediately after the last position that visits level $i-1$ and ending at the last position that visits level $i$.
In this way, $d$ is uniquely decomposed into $p$ sequences $\delta_1, \ldots, \delta_p \in \Dp$, of total 
length $2t-p-2$, and the remaining sequence of the form $01^a2^b$ with
$a+b=p$ and $a>0$ and $b\geq 0$. 

Let $\SEQ(D')$ denote the set 
of finite sequences of sequences from $D'$. 
Let $\SEQ_{\geq 1}(D')$ denote the set 
of nonempty finite sequences of sequences from $D'$.  
Let 
$$\Dpp:= \{(0)\} \times \SEQ_{\geq 1}(D') \times \SEQ(D')\enspace.$$ 
Let $h$ be the function that maps a sequence 
$d=(d_1, \ldots, d_{2t-1}) \in \Dp \setminus \{(0)\}$ to the triple
$((0), (\delta_1,\ldots, \delta_a), (\delta_{a+1},
\ldots, \delta_{a+b}))$, where $a, b$, and the $\delta_{i}$'s are defined 
as above. 
Observe that $h$ is a bijection between $D'\setminus
\{(0)\}$ and $\Dpp$. 

Let $C_{t,q}$ be the number of elements of $\Dpp$ with $t$ zeroes and
$q$ substrings 01. Define the formal power series 
$$C(z,y):= \sum_{t=0}^{\infty}\sum_{q=0}^\infty C_{t,q} z^t y^q\enspace.$$ 
Observe that $d$ and $h(d)$ have the same number of zeroes, 
for every $d\in D'\setminus \{(0)\}$. 
(Indeed, this is the reason for the leading $(0)$ in the definition of $\Dpp$.) 
Also, the total number of occurrences of 01 substrings in $h(d)$ is one less 
than in $d$. Thus $F_{t,q} = C_{t,q-1}$ for every $t \geq 1$, and 
$F(z,y) - z =y \cdot C(z,y)$. On the other
hand, it follows from the definition of $\Dpp$ that
$$C(z,y)= z \left(\sum_{i \geq 1}  F(z,y)^i \right) \left(\sum_{i \geq 0}  F(z,y)^i \right)
= z \left( \frac{F(z,y)}{1-F(z,y)} \right) \left(
  \frac{1}{1-F(z,y)} \right)\enspace.$$
This justifies that $F(z,y)$ satisfies the following equation:
$$F(z,y)= z+ z y \frac{F(z,y)}{(1-F(z,y))^2}\enspace.$$
In particular,
$$B(z)= z \left(1 + \Delta^{-1} \frac{B(z)}{(1-B(z) )^2}\right)\enspace.$$

Let $\phi(u)= 1+ \Delta^{-1} u / (1-u)^2$. Then $B(z)$ is the formal solution of the  equation $B(z)= z \phi(B(z))$.  It is straightforward to check that $\phi(u)$ satisfies the assumptions of Corollary \ref{cor:inverse_asymptotics} below. The radius of convergence of $\phi(u)$ is 1. Choose $\tau_0= 1-\Delta^{-1/3}$. Then 
$\tau_0 \frac{\phi'(\tau_0)}{\phi(\tau_0)}\neq 1$ and $\frac{\phi(\tau_0)}{\tau_0} = d$. 
Hence, Corollary \ref{cor:inverse_asymptotics} below implies that
$[z^t] B(z) = o(d^t)$. 
This completes the proof of the claim.
\end{proof} 

Returning to the proof of Theorem~\ref{thm:first_theorem_made_precise}, 
Claim~\ref{TheClaim} implies $|\mathcal{R}_t| = o( d^{t}  \Delta^{2t})$. Thus, if $t$ is large enough, then $|\mathcal{A}_t|$ is strictly
smaller than $\ell^t$, implying that there is at least one vector
$(c_1, \ldots, c_t)$ among the $\ell^t$ vectors in $[1,\ell]^t$ on
which the algorithm succeeds. Therefore $G$ admits a nonrepetitive
$L$-colouring.
\end{proof} 

The following results of Flajolet and Sedgewick~\cite{FS09} were used above.

\begin{theorem}[Proposition IV.5 from \cite{FS09}]
Let $\phi$ be a function analytic at 0, having nonnegative Taylor coefficients, and such that $\phi(0)\neq 0$. Let $R\leq +\infty$ be the radius of convergence of the series representing $\phi$ at 0. Under the condition, $$\lim_{x \to R^-} \frac{x \phi'(x)}{\phi(x)} > 1$$ there exists a unique solution $\tau \in (0,R)$ of the characteristic equation $\frac{\tau \phi'(\tau)}{\phi(\tau)}=1$. Then, the formal solution $y(z)$ of the equation $y(z)= z \phi(y(z))$ is analytic at 0 and has radius of convergence $\rho=\frac{\tau}{\phi(\tau)}$. 
\end{theorem}

For a function $\phi$ with nonnegative coefficients, the function $\frac{\tau}{\phi(\tau)}$ is concave in $(0,R)$. Then $\tau$ is the point in the interval $(0,R)$ that maximizes $\frac{\tau}{\phi(\tau)}$.  Thus, for any $\tau_0 \in (0,R)$ for which $\frac{\tau_0 \phi'(\tau_0)}{\phi(\tau_0)}\neq 1$, we have $\rho > \frac{\tau_0}{\phi(\tau_0)}$ which implies that $[z^n] y(z)= o(( \frac{\tau_0}{\phi(\tau_0)})^{-n})$.

\begin{corollary}
\label{cor:inverse_asymptotics}
Let $\phi$ be a function analytic at 0, having nonnegative Taylor coefficients, and such that $\phi(0)\neq 0$. Let $R\leq +\infty$ be the radius of convergence of $\phi$. Assume that  $\phi(R^-)= +\infty$ and let $y(z)$ be the formal solution of the equation  $y(z)= z \phi(y(z))$. Then for any $\tau_0 \in (0,R)$ for which $\frac{\tau_0 \phi'(\tau_0)}{\phi(\tau_0)}\neq 1$,
$$[z^n] y(z)= o\Big(\Big( \frac{\phi(\tau_0)}{\tau_0}\Big)^{n}\Big)\enspace.$$
\end{corollary}

\section{Subdivision Proof}
 
We now begin the proof of Theorem~\ref{thm:Subdivision}.  A sequence
$(s_{1}, \dots, s_{q})$ of positive integers is \DEF{$c$-spread} if
each entry equal to $1$ can be mapped to an entry greater than $1$
such that for each $i\in [1, q]$ with $s_{i} \geq 2$, there are
at least $\lceil c \log s_{i} \rceil$ entries, either all immediately
before $s_{i}$ or all immediately after $s_{i}$, that are equal to $1$
and are mapped to $s_{i}$.

\begin{lemma}
  \label{lem:spread}
  Fix $\eps>0$. Let $w :=(1+\eps)^{-1/2}<1$.  Let $c \in \N$ be such
  that $2^{2/c} \leq 1 + \eps$ and $w^{c} \leq \frac{\eps}{2}(1-w)$.
  Then for each $q \geq 1$ the number of distinct $c$-spread sequences
  of length $q$ is at most $(1+\eps)^q$.
\end{lemma}

\begin{proof}
  The proof is by induction on $q$.  Let $f(q)$ be the number of
  $c$-spread sequences of length $q$.  The claim holds when $q\leq c$
  since the sequence $(1, \dots, 1)$ of length $q$ is the only $c$-spread
  sequence of length $q$ in that case.

  Now assume that $q \geq c + 1$. Here are three ways of obtaining
  $c$-spread sequences of length $q$ from shorter ones:
  \begin{enumerate}
  \item If $(s_{1}, \dots, s_{q-1})$ is $c$-spread then so is $(1,
    s_{1}, \dots, s_{q-1})$.
  \item If $r \in \N$ such that $r \geq 2$ and $\lceil c \log r \rceil
    = q - 1$ then the two sequences $(1, \dots, 1, r)$ and
    $(r, 1, \dots, 1)$ of length $q$ are $c$-spread.
  \item If $r \in \N$ such that $r \geq 2$ and $z:=\lceil c \log r
    \rceil \leq q - 2$, and if $(s_{1}, \dots, s_{q - z -1})$ is a
    $c$-spread sequence, then the two sequences $(1, \dots,
    1, r, s_{1}, \dots, s_{q - z -1})$ and $(r, 1, \dots, 1, s_{1},
    \dots, s_{q - z -1})$ of length $q$ are $c$-spread.
  \end{enumerate}
  It is not difficult to see that each $c$-spread sequence of length
  $q$ can be obtained using the three constructions above.  Notice
  that if $z, r \in \N$ are such that $r \geq 2$ and $z=\lceil c \log
  r \rceil$, then in particular $z \geq c$ and $r \leq 2^{z/c}$.
  Letting $f(0) := 1$, we deduce that
  \begin{align*}
    f(q) &\leq f(q-1) + 2\sum_{z = c}^{q-1} 2^{z/c} f(q - z - 1) \\
    &\leq (1+\eps)^{q-1} + 2 \sum_{z = c}^{q-1} (1+\eps)^{z/2} (1 + \eps)^{q-z-1} \\
    &= (1+\eps)^{q-1} + 2(1+\eps)^{q-1}\sum_{z = c}^{q-1} w^{z} \\
    &\leq  (1+\eps)^{q-1} + 2(1+\eps)^{q-1}\sum_{z = c}^{\infty} w^{z} \\
    &=  (1+\eps)^{q-1} + 2(1+\eps)^{q-1} \frac{w^{c}}{1-w} \\
    &\leq (1+\eps)^{q-1} + \eps(1+\eps)^{q-1} \\
    &= (1 + \eps)^{q}.\qedhere
  \end{align*}
\end{proof}
  
A Dyck word $d$ is \DEF{special} if $d$ does not contain $0110110$ as a substring. 
The following crude upper bound on the number of such words   
will be used in our proof of Theorem~\ref{thm:Subdivision}.

\begin{lemma}
\label{lem:special_Dyck_words}
The number of special Dyck words of length $2t$ is at most $3.992^{t + 1}$. 
\end{lemma} 

\begin{proof}
For $q \geq 1$, let $g(q)$ be the number of binary words of length $q$ not containing $0110110$. Let $\xi := (2^{7} - 1)^{1/7}$. Then $g(q) \leq 2^{q} \leq \xi^{q+1}$ for $q \in [1, 7]$, and $g(q) \leq \xi^{7} \cdot g(q - 7)$ for $q \geq 8$, since  such binary words cannot start with $0110110$. Thus $g(q) \leq \xi^{q+1}$ for all $q \geq 1$. Since every special Dyck word of length $2t$ is a binary word not containing $0110110$, it follows that the number of such Dyck words is at most $\xi^{2t+1} < 3.992^{t + 1}$.
\end{proof}
 
\noindent\textbf{Theorem~\ref{thm:Subdivision}.} {\em Let $G$ be a
subdivision of a graph $H$, such that each edge $vw\in E(H)$ is
subdivided at least
$\ceil{10^{5}\log(\deg(v)+1)}+\ceil{10^{5}\log(\deg(w)+1)}+2$ times in $G$.
Then $\tc(G)\leq 5$.}

\begin{proof}
  Let $n:= |V(G)|$.  Let $L'(v)$ denote a list of available colours
  for each vertex $v \in V(G)$, and assume all these lists have size
  $5$.  Let $\psi$ be an arbitrary precolouring of $G$ with
  precoloured set $Q:= V(H)$ and with $\psi(v)\in L'(v)$ for each $v\in Q$.   
  Fix an ordering of $V(G)$ such that
  $V(G)-Q$ precedes $Q$.

  Let $c:= 10^{5}$.  For each $v \in Q$, let
  $g(v):=\ceil{c\log(\deg(v)+1)}$ and let $M(v)$ be the set of
  vertices of $G$ at distance at most $g(v)+1$ from $v$. Thus
  $M(v)\cap M(w)=\varnothing$ for distinct vertices $v, w \in Q$; we
  say that $u\in V(G)-Q$ \DEF{belongs to} $v$ if $u \in M(v)$ and $v\in Q$. 

  For each edge $vw \in E(H)$, let $P_{vw}$ denote the path of $G$
  induced by the subdivision vertices introduced on the edge $vw$ in
  $G$.  Note that $v, w \notin V(P_{vw})$.  A set $X \subseteq V(G) -
  Q$ is \DEF{nice} if $X\neq \varnothing$ and, for each edge $vw \in
  E(H)$, the graph $P_{vw} - X$ is either connected or empty.  The
  \DEF{boundary $\partial(X)$} of a nice set $X$ is the set of
  vertices $y \in X$ such that $X-\{y\}$ is either nice or
  empty. Observe that $\partial(X)$ is always nonempty.

  Fix an arbitrary ordering of the edges in $E(H)$. 
  For each edge $vw \in E(H)$, 
  orient the path $P_{vw}$ from an arbitrarily chosen endpoint to the other. 
  If $Y$ is a set of consecutive vertices of a path $P_{vw}$ and 
  $x \in V(P_{vw}) - Y$, then $x$ is either
  \DEF{before} $Y$ or \DEF{after} $Y$, 
  depending on the orientation of $P_{vw}$. 
 
  Let $f$ be a priority function defined as follows: For
  every nice set $X$, let $vw$ be the first edge in the ordering
  of $E(H)$ such that $V(P_{vw}) \cap X \neq \emptyset$. If $V(P_{vw}) \subseteq X$, 
  then $V(P_{vw}) \subseteq \partial(X)$, and we let $f(X)$ be an 
  arbitrary vertex in $V(P_{vw})$. If $V(P_{vw}) - X \neq \emptyset$ and
  there is a vertex $x \in \partial(X) \cap V(P_{vw})$ before $V(P_{vw}) - X$ on $P_{vw}$, 
  then $x$ is uniquely determined, and we let $f(X) := x$. 
  If $V(P_{vw}) - X \neq \emptyset$ but there is no such vertex $x$, then we
  let $f(X)$ be the unique vertex in $\partial(X) \cap V(P_{vw})$ that is after 
  $V(P_{vw}) - X$ on $P_{vw}$. 

  For each $u\in V(G) - Q$, let $L(u)$ be the
  list $L'(u) - \{\phi(v)\}$ if $u$ belongs to $v\in Q$ and $\phi(v)
  \in L'(u)$, otherwise let $L(u)$ be obtained from $L'(u)$ by
  removing one arbitrary colour from $L'(u)$.  This defines a list
  $L(u)$ of available colours for each vertex $u \in V(G) - Q$, and
  all these lists have size $\ell := 4$.  Consider
  the algorithm on $G$ with the latter lists, with
  priority function $f$, and with precolouring $\psi$.  By the
  definition of the lists $L(u)$, if the algorithm succeeds on some
  input $(c_{1}, \dots, c_{t}) \in [1,\ell]^{t}$ then it produces a
  nonrepetitive $L'$-colouring of $G$.

  \begin{claim}
    \label{cl:nice}
    Let $t \geq 1$. For each vector $(c_{1}, \dots, c_{t}) \in
    \F_{t}$, all the sets appearing in the trace $\tr(c_{1}, \dots,
    c_{t})$ are nice.
  \end{claim}

  \begin{proof}
    The proof is by induction on $t$.  The claim is true for $t=1$
    since $X_1=V(G)-Q$ is nice. Now assume that $t \geq 2$.  Let
    $(c_{1}, \dots, c_{t}) \in \F_{t}$ and let $\tr(c_{1}, \dots,
    c_{t}) = (X_{1}, \dots, X_{t})$.  Then $(c_{1}, \dots, c_{t-1})
    \in \F_{t-1}$, and by induction the sets $X_{1}, \dots, X_{t-1}$
    are nice.  Let $v_{t-1} := f(X_{t-1})$.

    First suppose that $R(t-1) = \varnothing$. Then $X_{t} = X_{t-1} -
    \{v_{t-1}\}$, which is a nice set since
    $v_{t-1}\in \partial(X_{t-1})$ and $X_{t} \neq \varnothing$.

    Now assume that $R(t-1) = \lambda$ for some $\lambda \in \S$.  Using
    Lemma~\ref{lem:P}, let $P$ be the path of $G$ determined by $\lambda$,
    $X_{t-1}$, and $v_{t-1}$.  Let $P_{1}$ and $P_{2}$ denote the two
    halves of $P$, so that $v_{t-1} \in V(P_{2})$.  Then $X_{t} =
    X_{t-1} \cup (V(P_{2}) - Q)$.  Arguing by contradiction, suppose
    that $X_{t}$ is not nice. Then there exists $vw \in E(H)$ such
    that $P_{vw} - X_{t}$ has at least two components.  Let $x, y$ be
    two vertices in distinct components of $P_{vw} - X_{t}$ that are
    as close as possible on the path $P_{vw}$. Then the set $Z$ of
    vertices strictly between $x$ and $y$ on $P_{vw}$ is a subset of
    $X_{t}$. On the other hand, $Z \cap X_{t-1} = \varnothing$ since
    otherwise $x$ and $y$ would be in distinct components of $P_{vw} -
    X_{t-1}$, contradicting the fact that $X_{t-1}$ is nice.  Thus $Z
    \subseteq V(P_{2}) - Q$, and also $\partial(X_{t-1}) \cap Z =
    \varnothing$. Since $P_{2}$ is connected and avoids $x$ and $y$,
    we deduce that $Z = V(P_{2})$ (and thus $Q \cap V(P_{2}) =
    \varnothing$).  However, $v_{t-1} \in \partial(X_{t-1})$ and
    $v_{t-1} \in V(P_{2})= Z$, contradicting $\partial(X_{t-1}) \cap Z
    = \varnothing$.
  \end{proof}

  \begin{claim}
    \label{cl:sizeSubdiv}
    $\beta \leq 1.001$.
  \end{claim}

  \begin{proof}
    We need to show that $\alpha_{k} \leq 1.001^{k}$ for each
    $k \in [1,\floor{\tfrac{n}{2}}]$.  Fix $k \in [1,\floor{\tfrac{n}{2}}]$. Let
    $\W$ be the set of triples $(P, X, v)$ that may be
    considered by the algorithm in the uncolouring step, over all $t
    \geq 1$ and vectors $(c_{1}, \dots, c_{t}) \in \F_{t}$, such that
    $P$ has exactly $2k$ vertices.

    Observe that if $(P, X, v) \in \W$, then $X$ is a nice subset of
    $V(G) - (Q \cup V(P))$ by Claim~\ref{cl:nice}; also, $v\in V(P)$
    and $X \cup \{v\}$ is again a nice set.  By the definition of
    $\W$, every realisable sequence $\lambda \in \S$ of length $2k$ is `produced'
    by at least one triple in $\W$, in the sense that there exists
    $(P, X, v) \in \W$ such that $\lambda = \lambda(P, X, v)$.  We may assume that
    $\W$ is not empty, since otherwise $\alpha_{k}=0$, and we are
    trivially done.
 
    Let $(P, X, v) \in \W$ and let $\lambda(P, X, v) = (\lambda_{1},\dots,\lambda_{2k})$.  Let $v_{1}, \dots, v_{2k}$ be the vertices of $P$.
    Note that $P$ may contain vertices of $Q$.  Since $v$ is not in $Q$, it has degree $2$, and thus $\lambda_{1} \in \{1,2\}$.  (Note that
    we could have $\lambda_{1}=2$ if the neighbour of $v$ in $P$ is in $Q$.) We have $\lambda_{p}=-1$ for a unique $p \in [2, 2k]$. We claim that the
    sequence $\lambda':=(\lambda_{p-1},\dots,\lambda_{2},1,\lambda_{p+1},\dots,\lambda_{2k})$ obtained from $(\lambda_{1}, \dots, \lambda_{2k})$ by removing the $p$-th entry, reversing the $(\lambda_{1},\dots,\lambda_{p-1})$ prefix, and replacing     $\lambda_{1}$ by $1$, is $c$-spread.

    \textbf{Case 1.} $2k \leq c + 1$: Then $P$ has no vertex $u$ in
    $Q$, since otherwise $P$ would have at least $g(u)+2
    \geq\ceil{c\log 2}+2 > c + 1$ vertices. It follows that there is
    an edge $xy \in E(H)$ such that $P$ is a subpath of $P_{xy}$. Then
    $v$ must be an endpoint of $P$.  Indeed, if not, then the two
    neighbours of $v$ in $P$ are in distinct components of $P_{xy} -
    (X \cup \{v\})$, contradicting the fact that $X \cup \{v\}$ is
    nice. Clearly $\lambda_{i} = 1$ for each $i \in [2, 2k-1]$ and $\lambda_{2k} =
    -1$. If $v$ is an internal vertex of $P_{xy}$, then one of the two
    neighbours of $v$ is in $X$, and it follows that $\lambda_{1} = 1$. If
    $v$ is an endpoint of $P_{xy}$, then $\lambda_{1}$ is the index in the
    set $N(v)$ of the only neighbour $w$ of $v$ that is in $P_{xy}$.
    This index is always $1$ by our choice of the ordering of
    $V(G)$ (since elements of $V(H)$ come last in the order). 
    Hence we again have $\lambda_{1}=1$. Therefore
    $(\lambda_{1},\dots,\lambda_{2k})=(1,\dots,1,-1)$ and $\lambda'$ is the sequence of
    $2k-1$ ones, which is $c$-spread.

    \textbf{Case 2.} $2k \geq c + 2$: If $\lambda_i>1$ for some $i\in[2,p]$,
    then $v_i$ is in $Q$; in this case, our goal is to show that $\lambda'$
    contains $g(v_i)$ ones immediately before or after
    $\lambda_i$ that can be mapped to $\lambda_i$. 
Similarly, if $\lambda_{p+j}>1$ for some $j\in[1,q]$, then $w_j$
    is in $Q$; in this case, our goal is to show that $\lambda'$ contains
    $g(w_j)$ ones immediately before or after $\lambda_{p+j}$ that can be
    mapped to $\lambda_{p+j}$.

    Consider a vertex $u\in V(P)\cap Q$. By the definition of $L$, the
    colour assigned to $u$ is assigned to no vertex that belongs to
    $u$ (those in the set $M(u)$) when the algorithm considers the
    triple $(P,X,v)$.  At that
    stage, $P$ is repetitively coloured. Let $x$ be the unique  vertex at distance $k$
    from $u$ in $P$. Then $u$ and $x$ are assigned the same
    colour, and $x$ is not in $M(u)$.  Walk from $u$
    towards $x$ and stop after $g(u)+1$ steps. This defines a subpath
    $P'$ of $P$ consisting of exactly $g(u)+1$ vertices that belong to
    $u$, either all immediately before $u$ or all immediately after
    $u$ in $P$. Consider the following six possible values of $u$ and
    $P'$:

\begin{itemize}
\item     If $u=v_i$ and $P'=(v_{i+g(u)+1},v_{i+g(u)},\dots,v_{i+1})$ then
    $\lambda_{i+g(u)}=\lambda_{i+g(u)-1}\dots=\lambda_{i+1}=1$ and $\lambda'$ contains $g(u)$
    ones immediately before $\lambda_i$ that can be mapped to $\lambda_i$.

\item     If $u=v_i$ and $P'=(v_{i-1},v_{i-2},\dots,v_{i-g(u)-1})$ and
    $i-g(u)-1\neq 1$, then $\lambda_{i-1}=\lambda_{i-2}=\dots=\lambda_{i-g(u)}=1$ and
    $\lambda'$ contains $g(u)$ ones immediately after $\lambda_i$ that can be
    mapped to $\lambda_i$.

\item     If $u=w_j$ and $P'=(w_{j+1},w_{j+2},\dots,w_{j+g(u)+1})$ then
    $\lambda_{p+j+1}=\lambda_{p+j+2}=\dots=\lambda_{p+j+g(u)}=1$ and $\lambda'$ contains
    $g(u)$ ones immediately after $\lambda_{p+j}$ that can be mapped to
    $\lambda_{p+j}$.

\item     If $u=w_j$ and $P'=(w_{j-g(u)-1},w_{j-g(u)},\dots,w_{j-1})$ then
    $\lambda_{p+j-g(u)}=\lambda_{p+j-g(u)+1}=\dots=\lambda_{p+j-1}=1$ and $\lambda'$ contains
    $g(u)$ ones immediately before $\lambda_{p+j}$ that can be mapped to
    $\lambda_{p+j}$.

\item     If $u=v_i$ and
    $P'=(v_{i-1},v_{i-2},\dots,v_1=w_1,w_2,\dots,w_{g(u)-i+3})$ then
    $\lambda_{i-1}=\lambda_{i-2}=\dots=\lambda_2=1$ and
    $\lambda_{p+1}=\lambda_{p+2}=\dots=\lambda_{p+g(u)-i+2}=1$, implying that
    $$\lambda_{i-1},\lambda_{i-2},\dots,\lambda_2,1,\lambda_{p+1},\lambda_{p+2},\dots,\lambda_{p+g(u)-i+1}$$
    is a sequence of $g(u)$ ones immediately after $\lambda_i$ in $\lambda'$ that
    can be mapped to $\lambda_i$.

\item     If $u=w_j$ and
    $P'=(v_{g(u)-j+3},v_{g(u)-j+2},\dots,v_1=w_1,w_2,\dots,w_{j-1})$
    then $\lambda_{g(u)-j+2}=\lambda_{g(u)-j+1}=\dots=\lambda_2=1$ and
    $\lambda_{p+1}=\lambda_{p+2}=\dots=\lambda_{p+j-1}=1$, implying that
    $$\lambda_{g(u)-j+2},\lambda_{g(u)-j+1},\dots,\lambda_2,1,\lambda_{p+1},\lambda_{p+2},\dots,\lambda_{p+j-1}$$
    is a sequence of $g(u)$ ones immediately before $\lambda_{p+j}$ in $\lambda'$
    that can be mapped to $\lambda_{p+j}$.
\end{itemize}

    Hence $\lambda'$ is $c$-spread, as claimed.

    Therefore $(\lambda_{1},\dots,\lambda_{2k})$ is obtained from a $c$-spread
    sequence $\lambda'$ of length $2k-1$ by choosing an index
    $p\in[1,2k-1]$, inserting $-1$ between the $p$-th and $(p+1)$-th
    entries, reversing the prefix of $\lambda'$ up to the $p$-th entry, and
    possibly changing the first entry to a $2$.  Hence the number of
    realisable sequences in $\S$ of length $2k$ is at most $2(2k-1)$ times the number of $c$-spread sequences of length $2k-1$.     Let $\eps := 0.0002$. Then $\eps$ and $c$ satisfy the
    hypotheses of Lemma~\ref{lem:spread}, and we deduce from that
    lemma that
$$
\alpha_{k} \leq (4k-2) \cdot (1+\eps)^{2k-1} \leq 4k (1+\eps)^{2k}
\leq 1.001^{k}\enspace.
$$ 
(The rightmost inequality holds because $2k \geq c$ and $2c
(1+\eps)^{c} \leq 1.001^{c/2}$.)
\end{proof}

Next we show that every realisable Dyck word is special. 
Consider a word $d \in \DD_{t}$ for some $t \geq 1$, and 
let $R\in \RR_{t}$ be a record such that $D(R) = d$. Suppose that
$d$ contains $0110110$ as a subsequence. Then  
there is an index $i \in [1, t-2]$ such that 
$|R(i)| = |R(i+1)| = 4$. 
Fix an arbitrary vector $(c_{1}, \dots, c_{t}) \in \F_{t, R}$, and let $(P,X,v)$ and $(P',X',v')$ 
be the triples such that $\lambda(P,X,v)=R(i)$ and $\lambda(P',X',v')=R(i+1)$, respectively, 
in the execution of the algorithm on input $(c_{1}, \dots, c_{t})$.   
Then $P$ contains no vertex from $Q$, since otherwise $P$ would need to have at least
$c+2 > 4$ vertices, as explained in Case~1 of the proof of Claim~\ref{cl:sizeSubdiv}. 
Since our ordering of $V(G)$ puts vertices in $V(G)-Q$ before those in $Q$, 
and since $X$ is nice, 
it follows that $\lambda(P,X,v)=R(i)=(1,1,1,-1)$.  
By the same argument $\lambda(P',X',v')=R(i+1)=(1,1,1,-1)$. 
Let $v_{1}, \dots, v_{4}$ denote the vertices of $P$, with $v_{4}=v$. 
Then in the $i$-th iteration of the while-loop of the algorithm, immediately after colouring
$v_{4}$, we have $\phi(v_{1})=\phi(v_{3})$ and $\phi(v_{2})=\phi(v_{4})$. 
Vertices $v_{3}$ and $v_{4}$ are subsequently uncoloured. 
Thus $X' = X \cup \{v_{3}\}$. 

By our choice of the priority function $f$, we have $f(X')=v'=v_{3}$. 
Indeed, $f(X)=v_{4}$ and $v_{1}, v_{2} \in V(P_{vw}) - X'$, where $vw\in E(H)$ is 
the edge such that $P \subseteq P_{vw}$. In particular, $v_{3} \in \partial(X')$ 
and $V(P_{vw}) - X' \neq \emptyset$. Thus  
either $v_{4}$ is before $V(P_{vw}) - X$ on $P_{vw}$, in which case $v_{3}$ is 
before $V(P_{vw}) - X'$ on $P_{vw}$, implying $f(X')=v_{3}$; or $v_{4}$ 
is after $V(P_{vw}) - X$ on $P_{vw}$, in which case $v_{3}$ is 
after $V(P_{vw}) - X'$ {\em and} there is no vertex in $\partial(X')$ 
before $V(P_{vw}) - X'$ on $P_{vw}$, implying again $f(X')=v_{3}$. 

It follows that the vertices of $P'$ are $v_{0}, v_{1}, v_{2}, v_{3}$ in order, 
where $v_{0} \in V(P_{vw})$ (and obviously $v_{0}\neq v_{4}$). 
In the $(i+1)$-th iteration of the while-loop, immediately  
after colouring $v_{3}$ (=$v'$), we have $\phi(v_{0}) = \phi(v_{2})$ and 
$\phi(v_{1}) = \phi(v_{3})$. However,  the colours of 
$v_{0}, v_{1}, v_{2}$ have not changed since the beginning of the $i$-th iteration, 
and $\phi(v_{1})=\phi(v_{3})$ at that point. This imples that
$P'$ was already repetitively coloured at the start of the $i$-th iteration, a contradiction. 
Therefore, realisable Dyck words are special, as claimed. 

Let $t \geq 1$, let $m:= \min\{n, t\}$, 
let $i\in [1, m]$ and let $(d_{1}, \dots, d_{2t})$ be a realisable Dyck word of
  length $2t$ such that $d_{2t-j} = 1$ for each $j\in [0, i-1]$. Let 
  $q$ be the number of maximal subsequences of consecutive ones in $(d_{1},
  \dots, d_{2t-i})$. If $q \geq 1$ then let $k_{1}, \dots, k_{q}$ be the
  lengths of these sequences, in order.  If $q \geq 1$, then
  $\sum_{j=1}^{q}k_{j}\leq t -i$, and we deduce that there are at most
  $$\alpha_{k_{1}} \alpha_{k_{2}} \cdots \alpha_{k_{q}} \leq
  \beta^{k_{1}} \beta^{k_{2}} \cdots \beta^{k_{q}} \leq \beta^{t-i}
  \leq \beta^{t}$$ distinct records $R \in \RR_{t}$ with $z(R)=i$ such
  that $D(R)=(d_1,\dots,d_{2t})$.  If $q = 0$, then there is at most
  $1 \leq \beta^{t}$ records $R \in \RR_{t}$ with $z(R)=i$ such that
  $D(R)=(d_1,\dots,d_{2t})$. 
  (Note that here we use the fact that $\beta \geq 1$.)

  Since for each $i\in [1, m]$, there are at most $\beta^{t}$ distinct
  records $R \in \RR_{t}$ with $z(R)=i$ that have the same Dyck word
  $D(R)$, and since there are exactly $|\DD_{t}|$ distinct  realisable special Dyck words of
  length $2t$, it follows that $|\RR_{t}| \leq m|\DD_{t}|\beta^{t} \leq
  n|\DD_{t}|\beta^{t}$.
Using Claim~\ref{cl:sizeSubdiv} and
Lemmas~\ref{lem:lossless} and~\ref{lem:special_Dyck_words}, we obtain 
$$
|\F_{t}| = |\A_{t}| \leq n(\ell+1)^{n}|\DD_{t}| \beta^{t} \leq
n(\ell+1)^{n} 3.992^{t+1} 1.001^{t}
< n(\ell+1)^{n} 3.996^{t+1} \enspace.
$$
Hence, if $t$ is sufficiently large then $|\F_{t}| < \ell^{t}$, implying that 
there is at least one vector $(c_{1},\dots,c_{t})$ among the
$\ell^{t}$ many vectors in $[1,\ell]^{t}$ on which the algorithm
succeeds.  Therefore $G$ admits a nonrepetitive $L'$-colouring.
\end{proof}

Note that we made no effort to optimise the constant $10^{5}$ in the
proof of Theorem~\ref{thm:Subdivision}.

\section{Pathwidth Proofs}

The proof of Theorem~\ref{thm:Pathwidth} depends on the following
lemma of independent interest.

\begin{lemma}
  \label{lem:Blocks}
  Let $B_1,\dots,B_m$ be pairwise disjoint sets of vertices in a graph
  $G$, such that no two vertices in distinct $B_i$ are adjacent. Let
  $H$ be the graph obtained from $G$ by deleting $B_i$ and adding a
  clique on $N_G(B_i)$ for each $i\in[1,m]$. Then
  \begin{align*}
    \pi(G)&\leq\pi(H)+\max_i\pi(G[B_i])\enspace.
  \end{align*}
\end{lemma}

\begin{proof}
  Nonrepetitively colour $G[B_1\cup\dots\cup B_m]$ with
  $\max_i\pi(G[B_i])$ colours. Nonrepetitively colour $H$ with a
  disjoint set of $\pi(H)$ colours. Suppose on the contrary that $G$
  contains a repetitively coloured path $P$. Let $P'$ be the set of
  vertices in $P$ that are in $H$, ordered according to $P$. Then
  $P'\neq\varnothing$, as otherwise $P$ is contained in some $B_i$,
  implying $B_i$ contains a repetitively coloured path. Consider a
  maximal subpath $S$ in $P$ that is not in $H$. So $S$ was deleted
  from $P$ in the construction of $P'$. Since no two vertices in
  distinct $B_i$ are adjacent, $S$ is contained in a single set
  $B_i$. Thus the vertices in $P$ immediately before and after $S$ (if
  they exist) are in $N_G(B_i)$, and are thus adjacent in $H$. Hence
  $P'$ is a path in $H$. Since the vertices in $B_1\cup\dots\cup B_m$
  receive distinct colours from the vertices in $H$, the path $P'$ is
  repetitively coloured. This contradiction proves that $G$ is
  nonrepetitively coloured.
\end{proof}

The next lemma provides a useful way to think about graphs of bounded
pathwidth. Let $G\cdot K_\width$ denote the \DEF{lexicographical product}
of a graph $G$ and the complete graph $K_\width$. That is, $G\cdot K_\width$ is
obtained by replacing each vertex of $G$ by a copy of $K_\width$, and
replacing each edge of $G$ by a copy of $K_{\width,\width}$.

\begin{lemma}
  \label{lem:Attack}
  Every graph $G$ with pathwidth $\width$ contains pairwise disjoint sets
  $B_1,\dots,B_m$ of vertices, such that: \vspace*{-2ex}
  \begin{itemize}
  \item no two vertices in distinct $B_i$ are adjacent,
  \item $G[B_i]$ has pathwidth at most $\width-1$ for each $i\in[1,m]$, and
  \item if $H$ is the graph obtained from $G$ by deleting $B_i$ and
    adding a clique on $N_G(B_i)$ for each $i\in[1,m]$, then $H$ is a
    subgraph of $P_m\cdot K_{\width+1}$.
  \end{itemize}
\end{lemma}

\begin{proof}
  Consider a path decomposition $\mathcal{D}$ of $G$ with width $\width$.
  Let $X_1,\dots,X_m$ be the set of bags in $\mathcal{D}$, such that
  $X_1$ is the first bag in $\mathcal{D}$, and for each $i\geq2$, the
  bag $X_i$ is the first bag in $\mathcal{D}$ that is disjoint from
  $X_{i-1}$. Thus $X_1,\dots,X_m$ are pairwise disjoint.  For
  $i\in[1,m]$, let $B_i$ be the set of vertices that only appear in
  bags strictly between $X_i$ and $X_{i+1}$ (or strictly after $X_m$
  if $i=m$). By construction, each such bag intersects $X_i$. Hence
  $G[B_i]$ has pathwidth at most $\width-1$. Since each $X_i$ separates
  $B_{i-1}$ and $B_{i+1}$ (for $i\neq m$), no two vertices in distinct
  $B_i$ are adjacent. Moreover, the neighbourhood of $B_i$ is
  contained in $X_i\cup X_{i+1}$ (or $X_i$ if $i=m$).  Hence the graph
  $H$ (defined above) has vertex set $X_1\cup\cdots\cup X_{m}$ where
  $X_i\cup X_{i+1}$ is a clique for each $i\in[1,m-1]$.  Since
  $|X_i|\leq \width+1$, the graph $H$ is a subgraph of $P_{m}\cdot
  K_{\width+1}$.
\end{proof}

\begin{proof}[Proof of Theorem~\ref{thm:Pathwidth}] 
  We proceed by induction on $\width\geq 0$.  Every graph with pathwidth 0
  is edgeless, and is thus nonrepetitively 1-colourable, as desired.
  Now assume that $G$ is a graph with pathwidth $\width\geq 1$.  Let
  $B_1,\dots,B_m$ be the sets that satisfy Lemma~\ref{lem:Attack}. Let
  $H$ be the graph obtained from $G$ by deleting $B_i$ and adding a
  clique on $N_G(B_i)$ for each $i\in[1,m]$. Then $H$ is a subgraph of
  $P_{m+1}\cdot K_{\width+1}$, which is nonrepetitively $4(\width+1)$-colourable
  by a theorem of \citet{KP-DM08}\footnote{Say $V_1,\dots,V_t$ is a
    partition of $V(G)$ such that for all $i\in[1,t]$, we have
    $N_G(V_i)\subseteq V_{i-1}\cup V_{i+1}$ and $N_G(V_i)\cap V_{i-1}$
    is a clique.  \citet{KP-DM08} proved that $\pi(G)\leq
    4\max_i\pi(G[V_i])$. Clearly $P_m\cdot K_{\width+1}$ has such a
    partition with each $V_i$ a $(\width+1)$-clique. Thus $\pi(P_m\cdot
    K_{\width+1})\leq 4(\width+1)$.}.  By induction, $\pi(G[B_i])\leq
  2(\width-1)^2+6(\width-1)-4$.  By Lemma~\ref{lem:Blocks},
  $\pi(G)\leq\pi(H)+\max_i\pi(G[B_i]) \leq 4(\width+1)+ 2(\width-1)^2+6(\width-1)-4=
  2\width^2+6\width-4$. This completes the proof.
\end{proof}

\begin{proof}[Proof of Theorem~\ref{thm:StarPathwidth}]
  We proceed by induction on $\width\geq 0$.  Every graph with pathwidth 0 is edgeless, and is thus star 1-colourable, as desired.  Now assume that $G$ is a graph with pathwidth $\width\geq 1$.  We may assume that $G$ is connected. Let $G'$ be an interval graph containing $G$ as a spanning subgraph and with $\maxclique(G')=\width+1$. Let $I(v)$ be the interval representing each vertex $v$. Let $X$ be an inclusion-wise minimal set of vertices in $G'$ such that for every
  vertex $w$,
  \begin{equation}
    \label{eqn:Condition}
    I(w)\subseteq \bigcup\{I(v):v\in X\}\enspace.
  \end{equation}
  The set $X$ exists since $X=V(G)$ satisfies \eqref{eqn:Condition}.
  It is easily seen that $G[X]$ is an induced path, say
  $(x_1,\dots,x_n)$.  Colour $x_i$ by $i\bmod{3}$ (in $\{0,1,2\}$).
  Observe that $G'[X]$ is star 3-coloured. By \eqref{eqn:Condition},
  the subgraph $G'-X$ is an interval graph with $\maxclique(G'-X)\leq
  \width$. Thus, by induction, $G'-X$ is star-colourable with colours
  $\{3,4,\dots,3\width\}$. Suppose on the contrary that $G'$ contains a
  2-coloured path $(u,v,w,x)$. First suppose that $u$ is in $X$. Then
  $w$ is also in $X$. If $v$ is also in $X$ then so is $x$, which
  contradicts the fact that $G'[X]$ is star-coloured. So $v\not\in
  X$. Since $u$ and $w$ receive the same colour, there are at least
  two vertices $p$ and $q$ between $u$ and $w$ in the path
  $G'[X]$. Thus replacing $p$ and $q$ by $v$ gives a shorter path that
  satisfies \eqref{eqn:Condition}. This contradiction proves that
  $u\not\in X$. By symmetry $x\not\in X$.  Since $X$ and $G'-X$ are
  assigned disjoint sets of colours, $v\not\in X$ and $w\not\in
  X$. Hence $(u,v,w,x)$ is a 2-coloured path in $G'-X$, which is the
  desired contradiction. Hence $G'$ is star-coloured with $3\width+1$
  colours.
\end{proof}

\section{Open problems}

We conclude with a number of open problems:

\begin{itemize}

\item Whether there is a relationship between nonrepetitive
  choosability and pathwidth is an interesting open problem. The graphs with pathwidth 1
  (i.e., caterpillars) are nonrepetitively $\ell$-choosable for some
  constant $\ell$; see Appendix~\ref{Caterpillars}. Is every graph (or tree) with pathwidth 2
  nonrepetitively $\ell$-choosable for some constant $\ell$?

\item Except for a finite number of examples, every cycle is
  nonrepetitively $3$-colourable \citep{Currie-EJC02}.  Every cycle is
  nonrepetitively 5-choosable. (\emph{Proof.}  Precolour one vertex,
  remove this colour from every other list, and apply the
  nonrepetitive 4-choosability result for paths.) Is every cycle
  nonrepetitively $4$-choosable? Which cycles are nonrepetitively
  $3$-choosable?

\item Does every graph have a nonrepetitively $4$-choosable
  subdivision? Even $3$-choosable might be possible.

\item Is there a function $f$ such that every graph $G$ has a
  nonrepetitively $\mathcal{O}(1)$-colourable subdivision with
  $f(\pi(G))$ division vertices per edge?

\item Is there a function $f$ such that $\pi(G/M)\leq f(\pi(G))$ for
  every graph $G$ and for every matching $M$ of $G$, where $G/M$ denotes the graph
  obtained from $G$ by contracting the edges in $M$? This would
  generalise a result of \citet{NOW} about subdivisions (when each
  edge in $M$ has one endpoint of degree 2).

\item Is there a polynomial-time Monte Carlo algorithm that
  nonrepetitively $\mathcal{O}(\Delta^2)$-colours a graph with maximum
  degree $\Delta$? \citet{HSS-JACM} show that
  $\mathcal{O}(\Delta^{2+\eps})$ colours suffice for all fixed
  $\eps>0$; also see \citep{KS-STOC,CGH-SJC13} for related results. Note
  that testing whether a given colouring of a graph is nonrepetitive
  is co-NP-complete, even for 4-colourings \citep{DS09}.

\end{itemize}

\subsection*{Acknowledgement} Thanks to Pascal Ochem who pointed out an error in an earlier version of the paper, and to the 
two anonymous referees for their helpful comments.   


\begin{thebibliography}{48}
\providecommand{\natexlab}[1]{#1}
\providecommand{\url}[1]{\texttt{#1}}
\providecommand{\urlprefix}{}
\expandafter\ifx\csname urlstyle\endcsname\relax
  \providecommand{\doi}[1]{doi:\discretionary{}{}{}#1}\else
  \providecommand{\doi}{doi:\discretionary{}{}{}\begingroup
  \urlstyle{rm}\Url}\fi

\bibitem[{Albertson et~al.(2004)Albertson, Chappell, Kierstead, K{\"u}ndgen,
  and Ramamurthi}]{Albertson-EJC04}
\textsc{Michael~O. Albertson, Glenn~G. Chappell, Hal~A. Kierstead, Andr\'{e}
  K{\"u}ndgen, and Radhika Ramamurthi}.
\newblock Coloring with no 2-colored ${P}_4$'s.
\newblock \emph{Electron. J. Combin.}, 11 \#R26, 2004.
\newblock
  \urlprefix\url{http://www.combinatorics.org/Volume_11/Abstracts/v11i1r26.html}.
\newblock \msn{2056078}.

\bibitem[{Alon and Grytczuk(2008)}]{BreakingRhythm}
\textsc{Noga Alon and Jaros{\l}aw Grytczuk}.
\newblock Breaking the rhythm on graphs.
\newblock \emph{Discrete Math.}, 308:1375--1380, 2008.
\newblock \doi{10.1016/j.disc.2007.07.063}.
\newblock \MSN{2392054}{2009c:05066}.

\bibitem[{Alon et~al.(2002)Alon, Grytczuk, Ha{\l}uszczak, and
  Riordan}]{AGHR-RSA02}
\textsc{Noga Alon, Jaros{\l}aw Grytczuk, Mariusz Ha{\l}uszczak, and Oliver
  Riordan}.
\newblock Nonrepetitive colorings of graphs.
\newblock \emph{Random Structures Algorithms}, 21(3-4):336--346, 2002.
\newblock \doi{10.1002/rsa.10057}.
\newblock \msn{1945373}.

\bibitem[{Bar{\'a}t and Czap(2013)}]{BC13}
\textsc{J{\'a}nos Bar{\'a}t and J\'ulius Czap}.
\newblock Facial nonrepetitive vertex coloring of plane graphs.
\newblock \emph{J. Graph Theory}, 74(1):115--121, 2013.
\newblock \doi{10.1002/jgt.21695}.

\bibitem[{Bar{\'a}t and Varj{\'u}(2007)}]{BV-NonRepVertex07}
\textsc{J{\'a}nos Bar{\'a}t and P{\'e}ter~P. Varj{\'u}}.
\newblock On square-free vertex colorings of graphs.
\newblock \emph{Studia Sci. Math. Hungar.}, 44(3):411--422, 2007.
\newblock \doi{10.1556/SScMath.2007.1029}.
\newblock \MSN{2361685}{2008j:05131}.

\bibitem[{Bar{\'a}t and Varj{\'u}(2008)}]{BV-NonRepEdge08}
\textsc{J{\'a}nos Bar{\'a}t and P{\'e}ter~P. Varj{\'u}}.
\newblock On square-free edge colorings of graphs.
\newblock \emph{Ars Combin.}, 87:377--383, 2008.
\newblock \MSN{2414029}{2009c:05068}.

\bibitem[{Bar{\'a}t and Wood(2008)}]{BaratWood-EJC08}
\textsc{J{\'a}nos Bar{\'a}t and David~R. Wood}.
\newblock Notes on nonrepetitive graph colouring.
\newblock \emph{Electron. J. Combin.}, 15:R99, 2008.
\newblock
  \urlprefix\url{http://www.combinatorics.org/Volume_15/Abstracts/v15i1r99.html}.
\newblock \msn{2426162}.

\bibitem[{Borodin(1979)}]{Borodin-DM79}
\textsc{Oleg~V. Borodin}.
\newblock On acyclic colorings of planar graphs.
\newblock \emph{Discrete Math.}, 25(3):211--236, 1979.
\newblock \doi{10.1016/0012-365X(79)90077-3}.
\newblock \msn{534939}.

\bibitem[{Bre{\v{s}}ar et~al.(2007)Bre{\v{s}}ar, Grytczuk, Klav{\v{z}}ar,
  Niwczyk, and Peterin}]{BGKNP-NonRepTree-DM07}
\textsc{Bo{\v{s}}tjan Bre{\v{s}}ar, Jaros{\l}aw Grytczuk, Sandi Klav{\v{z}}ar,
  Stanis{\l}aw Niwczyk, and Iztok Peterin}.
\newblock Nonrepetitive colorings of trees.
\newblock \emph{Discrete Math.}, 307(2):163--172, 2007.
\newblock \doi{10.1016/j.disc.2006.06.017}.
\newblock \msn{2285186}.

\bibitem[{Bre{\v{s}}ar and Klav{\v{z}}ar(2004)}]{BK-AC04}
\textsc{Bo{\v{s}}tjan Bre{\v{s}}ar and Sandi Klav{\v{z}}ar}.
\newblock Square-free colorings of graphs.
\newblock \emph{Ars Combin.}, 70:3--13, 2004.
\newblock \msn{2023057}.

\bibitem[{Chandrasekaran et~al.(2013)Chandrasekaran, Goyal, and
  Haeupler}]{CGH-SJC13}
\textsc{Karthekeyan Chandrasekaran, Navin Goyal, and Bernhard Haeupler}.
\newblock Deterministic {A}lgorithms for the {L}ov\'asz {L}ocal {L}emma.
\newblock \emph{SIAM J. Comput.}, 42(6):2132--2155, 2013.
\newblock \doi{10.1137/100799642}.

\bibitem[{Cheilaris et~al.(2010)Cheilaris, Specker, and Zachos}]{CSZ}
\textsc{Panagiotis Cheilaris, Ernst Specker, and Stathis Zachos}.
\newblock Neochromatica.
\newblock \emph{Comment. Math. Univ. Carolin.}, 51(3):469--480, 2010.
\newblock \urlprefix\url{http://www.dml.cz/dmlcz/140723}.
\newblock \msn{2741880}.

\bibitem[{Currie(2002)}]{Currie-EJC02}
\textsc{James~D. Currie}.
\newblock There are ternary circular square-free words of length {$n$} for
  {$n\geq 18$}.
\newblock \emph{Electron. J. Combin.}, 9(1), 2002.
\newblock
  \urlprefix\url{http://www.combinatorics.org/Volume_9/Abstracts/v9i1n10.html}.
\newblock \msn{1936865}.

\bibitem[{Currie(2005)}]{Currie-TCS05}
\textsc{James~D. Currie}.
\newblock Pattern avoidance: themes and variations.
\newblock \emph{Theoret. Comput. Sci.}, 339(1):7--18, 2005.
\newblock \doi{10.1016/j.tcs.2005.01.004}.
\newblock \msn{2142070}.

\bibitem[{Czerwi\'{n}ski and Grytczuk(2007)}]{CG-ENDM07}
\textsc{Sebastian Czerwi\'{n}ski and Jaros{\l}aw Grytczuk}.
\newblock Nonrepetitive colorings of graphs.
\newblock \emph{Electron. Notes Discrete Math.}, 28:453--459, 2007.
\newblock \doi{10.1016/j.endm.2007.01.063}.
\newblock \msn{2324051}.

\bibitem[{Dujmovi\'{c} et~al.(2011)Dujmovi\'{c}, Joret, Kozik, and
  Wood}]{NonRep_arxiv}
\textsc{Vida Dujmovi\'{c}, Gwena\"{e}l Joret, Jakub Kozik, and David~R. Wood}.
\newblock Nonrepetitive colouring via entropy compression.
\newblock 2011.
\newblock \arXiv{1112.5524}.

\bibitem[{Esperet and Parreau(2013)}]{EsperetParreau}
\textsc{Louis Esperet and Aline Parreau}.
\newblock Acyclic edge-coloring using entropy compression.
\newblock \emph{European J. Combin.}, 34(6):1019--1027, 2013.
\newblock \doi{10.1016/j.ejc.2013.02.007}.

\bibitem[{Fertin et~al.(2004)Fertin, Raspaud, and Reed}]{FRR-JGT04}
\textsc{Guillaume Fertin, Andr{\'e} Raspaud, and Bruce Reed}.
\newblock Star coloring of graphs.
\newblock \emph{J. Graph Theory}, 47(3):163--182, 2004.
\newblock \doi{10.1002/jgt.20029}.
\newblock \msn{2089462}.

\bibitem[{Fiorenzi et~al.(2011)Fiorenzi, Ochem, de~Mendez, and Zhu}]{FOOZ}
\textsc{Francesca Fiorenzi, Pascal Ochem, Patrice~Ossona de~Mendez, and Xuding
  Zhu}.
\newblock Thue choosability of trees.
\newblock \emph{Discrete Applied Math.}, 159(17):2045--2049, 2011.
\newblock \doi{10.1016/j.dam.2011.07.017}.
\newblock \msn{2832329}.

\bibitem[{Flajolet and Sedgewick(2009)}]{FS09}
\textsc{Philippe Flajolet and Robert Sedgewick}.
\newblock \emph{Analytic combinatorics}.
\newblock Cambridge University Press, 2009.

\bibitem[{Gr{\"u}nbaum(1973)}]{Grunbaum73}
\textsc{Branko Gr{\"u}nbaum}.
\newblock Acyclic colorings of planar graphs.
\newblock \emph{Israel J. Math.}, 14:390--408, 1973.
\newblock \doi{10.1007/BF02764716}.
\newblock \msn{0317982}.

\bibitem[{Grytczuk(2002)}]{Grytczuk-EJC02}
\textsc{Jaros{\l}aw Grytczuk}.
\newblock Thue-like sequences and rainbow arithmetic progressions.
\newblock \emph{Electron. J. Combin.}, 9(1):R44, 2002.
\newblock
  \urlprefix\url{http://www.combinatorics.org/Volume_9/Abstracts/v9i1r44.html}.
\newblock \msn{1946146}.

\bibitem[{Grytczuk(2007{\natexlab{a}})}]{Gryczuk-IJMMS07}
\textsc{Jaros{\l}aw Grytczuk}.
\newblock Nonrepetitive colorings of graphs---a survey.
\newblock \emph{Int. J. Math. Math. Sci.}, 74639, 2007{\natexlab{a}}.
\newblock \doi{10.1155/2007/74639}.
\newblock \msn{2272338}.

\bibitem[{Grytczuk(2007{\natexlab{b}})}]{Grytczuk}
\textsc{Jaros{\l}aw Grytczuk}.
\newblock Nonrepetitive graph coloring.
\newblock In \emph{Graph Theory in Paris}, Trends in Mathematics, pp. 209--218.
  Birkhauser, 2007{\natexlab{b}}.

\bibitem[{Grytczuk(2008)}]{Grytczuk-DM08}
\textsc{Jaros{\l}aw Grytczuk}.
\newblock Thue type problems for graphs, points, and numbers.
\newblock \emph{Discrete Math.}, 308(19):4419--4429, 2008.
\newblock \doi{10.1016/j.disc.2007.08.039}.
\newblock \msn{2433769}.

\bibitem[{Grytczuk et~al.(2013)Grytczuk, Kozik, and Micek}]{GKM11}
\textsc{Jaros{\l}aw Grytczuk, Jakub Kozik, and Piotr Micek}.
\newblock A new approach to nonrepetitive sequences.
\newblock \emph{Random Structures Algorithms}, 42(2):214--225, 2013.
\newblock \doi{10.1002/rsa.20411}.

\bibitem[{Grytczuk et~al.(2011)Grytczuk, Przyby{\l}o, and Zhu}]{GPZ}
\textsc{Jaros{\l}aw Grytczuk, Jakub Przyby{\l}o, and Xuding Zhu}.
\newblock Nonrepetitive list colourings of paths.
\newblock \emph{Random Structures Algorithms}, 38(1-2):162--173, 2011.
\newblock \doi{10.1002/rsa.20347}.
\newblock \msn{2768888}.

\bibitem[{Haeupler et~al.(2011)Haeupler, Saha, and Srinivasan}]{HSS-JACM}
\textsc{Bernhard Haeupler, Barna Saha, and Aravind Srinivasan}.
\newblock New constructive aspects of the {L}ov\'asz local lemma.
\newblock \emph{J. ACM}, 58(6):Art. 28, 2011.
\newblock \doi{10.1145/2049697.2049702}.

\bibitem[{Harant and Jendrol'(2012)}]{HJ-DM11}
\textsc{Jochen Harant and Stanislav Jendrol'}.
\newblock Nonrepetitive vertex colorings of graphs.
\newblock \emph{Discrete Math.}, 312(2):374--380, 2012.
\newblock \doi{10.1016/j.disc.2011.09.027}.
\newblock \msn{2852595}.

\bibitem[{Havet et~al.(2011)Havet, Jendro{\soft{l}}, Sot{\'a}k, and
  {\v{S}}krabu{\soft{l}}{\'a}kova}]{HJSS}
\textsc{Fr{\'e}d{\'e}ric Havet, Stanislav Jendro{\soft{l}}, Roman Sot{\'a}k,
  and Erika {\v{S}}krabu{\soft{l}}{\'a}kova}.
\newblock Facial non-repetitive edge-coloring of plane graphs.
\newblock \emph{J. Graph Theory}, 66(1):38--48, 2011.
\newblock \doi{10.1002/jgt.20488}.
\newblock \msn{2742187}.

\bibitem[{Jendrol and {\v{S}}krabul'{\'a}kov{\'a}(2009)}]{JS09}
\textsc{Stanislav Jendrol and Erika {\v{S}}krabul'{\'a}kov{\'a}}.
\newblock Facial non-repetitive edge colouring of semiregular polyhedra.
\newblock \emph{Acta Univ. M. Belii Ser. Math.}, 15:37--52, 2009.
\newblock \urlprefix\url{http://actamath.savbb.sk/acta1503.shtml}.
\newblock \msn{2589669}.

\bibitem[{Keszegh et~al.(2012)Keszegh, Patk\'os, and Zhu}]{Keszegh}
\textsc{Bal\'azs Keszegh, Bal\'azs Patk\'os, and Xuding Zhu}.
\newblock Nonrepetitive colorings of blow-ups of graphs.
\newblock 2012.
\newblock \arXiv{1210.5607}.

\bibitem[{Kolipaka and Szegedy(2011)}]{KS-STOC}
\textsc{Kashyap Kolipaka and Mario Szegedy}.
\newblock {M}oser and {T}ardos meet {L}ov{\'a}sz.
\newblock In \emph{Proc. 43rd Annual ACM Symposium on Theory of Computing (STOC
  '11)}, pp. 235--244. ACM, 2011.
\newblock \doi{10.1145/1993636.1993669}.

\bibitem[{Kolipaka et~al.(2012)Kolipaka, Szegedy, and Xu}]{Kolipaka}
\textsc{Kashyap Kolipaka, Mario Szegedy, and Yixin Xu}.
\newblock A sharper local lemma with improved applications.
\newblock In \emph{Approximation, Randomization, and Combinatorial
  Optimization. Algorithms and Techniques}, vol. 7408 of \emph{Lecture Notes in
  Computer Science}, pp. 603--614. 2012.
\newblock
  \urlprefix\url{http://www.springerlink.com/content/411621u1n0617681/}.

\bibitem[{Kozik and Micek(2013)}]{KozikMicek}
\textsc{Jakub Kozik and Piotr Micek}.
\newblock Nonrepetitive choice number of trees.
\newblock \emph{SIAM J. Discrete Math.}, 27(1):436--446, 2013.
\newblock \doi{10.1137/120866361}.
\newblock \arXiv{1207.5155}.

\bibitem[{K{\"u}ndgen and Pelsmajer(2008)}]{KP-DM08}
\textsc{Andre K{\"u}ndgen and Michael~J. Pelsmajer}.
\newblock Nonrepetitive colorings of graphs of bounded tree-width.
\newblock \emph{Discrete Math.}, 308(19):4473--4478, 2008.
\newblock \doi{10.1016/j.disc.2007.08.043}.
\newblock \msn{2433774}.

\bibitem[{Manin(2007)}]{Manin}
\textsc{Fedor Manin}.
\newblock The complexity of nonrepetitive edge coloring of graphs, 2007.
\newblock \arXiv{0709.4497}.

\bibitem[{Marx and Schaefer(2009)}]{DS09}
\textsc{D{\'a}niel Marx and Marcus Schaefer}.
\newblock The complexity of nonrepetitive coloring.
\newblock \emph{Discrete Appl. Math.}, 157(1):13--18, 2009.
\newblock \doi{10.1016/j.dam.2008.04.015}.
\newblock \msn{2479374}.

\bibitem[{Molloy and Reed(2002)}]{MR02}
\textsc{Michael Molloy and Bruce Reed}.
\newblock \emph{Graph colouring and the probabilistic method}, vol.~23 of \emph{Algorithms and Combinatorics}.
\newblock Springer, 2002.

\bibitem[{Moser and Tardos(2010)}]{MoserTardos}
\textsc{Robin~A. Moser and G{\'a}bor Tardos}.
\newblock A constructive proof of the general {L}ov\'asz local lemma.
\newblock \emph{J. ACM}, 57(2), 2010.
\newblock \doi{10.1145/1667053.1667060}.

\bibitem[{Ne{\v{s}}et{\v{r}}il and {Ossona de Mendez}(2003)}]{NesOdM-03}
\textsc{Jaroslav Ne{\v{s}}et{\v{r}}il and Patrice {Ossona de Mendez}}.
\newblock Colorings and homomorphisms of minor closed classes.
\newblock In \textsc{Boris Aronov, Saugata Basu, J\'{a}nos Pach, and Micha
  Sharir}, eds., \emph{Discrete and Computational Geometry, The Goodman-Pollack
  Festschrift}, vol.~25 of \emph{Algorithms and Combinatorics}, pp. 651--664.
  Springer, 2003.
\newblock \msn{2038495}.

\bibitem[{Ne{\v{s}}et{\v{r}}il et~al.(2011)Ne{\v{s}}et{\v{r}}il, {Ossona de
  Mendez}, and Wood}]{NOW}
\textsc{Jaroslav Ne{\v{s}}et{\v{r}}il, Patrice {Ossona de Mendez}, and David~R.
  Wood}.
\newblock Characterisations and examples of graph classes with bounded
  expansion.
\newblock \emph{European J. Combin.}, 33(3):350--373, 2011.
\newblock \doi{10.1016/j.ejc.2011.09.008}.
\newblock \msn{2864421}.

\bibitem[{Ochem and Pinlou(2013)}]{OchemPinlou}
\textsc{Pascal Ochem and Alexandre Pinlou}.
\newblock Application of entropy compression in pattern avoidance.
\newblock 2013.
\newblock \arXiv{1301.1873}.

\bibitem[{Pegden(2011)}]{Pegden11}
\textsc{Wesley Pegden}.
\newblock Highly nonrepetitive sequences: winning strategies from the local
  lemma.
\newblock \emph{Random Structures Algorithms}, 38(1-2):140--161, 2011.
\newblock \doi{10.1002/rsa.20354}.
\newblock \MSN{2768887}.

\bibitem[{Pezarski and Zmarz(2009)}]{PZ09}
\textsc{Andrzej Pezarski and Micha{\l} Zmarz}.
\newblock Non-repetitive 3-coloring of subdivided graphs.
\newblock \emph{Electron. J. Combin.}, 16(1):N15, 2009.
\newblock
  \urlprefix\url{http://www.combinatorics.org/Volume_16/Abstracts/v16i1n15.html}.
\newblock \msn{2515755}.

\bibitem[{Przyby{\l}o(2013)}]{Przybyo}
\textsc{Jakub Przyby{\l}o}.
\newblock On the facial {T}hue choice index via entropy compression.
\newblock \emph{J. Graph Theory}, 2013.
\newblock \doi{10.1002/jgt.21781}.
\newblock \arXiv{1207.0964}.

\bibitem[{Przyby{\l}o et~al.(2013)Przyby{\l}o, Schreyer, and
  {\v{S}}krabu{\soft{l}}{\'a}kova}]{Przybyo13}
\textsc{Jakub Przyby{\l}o, Jens Schreyer, and Erika
  {\v{S}}krabu{\soft{l}}{\'a}kova}.
\newblock On the facial {Thue} choice number of plane graphs via entropy
  compression method.
\newblock 2013.
\newblock \arXiv{1308.5128}.

\bibitem[{Thue(1906)}]{Thue06}
\textsc{Axel Thue}.
\newblock {\"U}ber unendliche {Z}eichenreihen.
\newblock \emph{Norske Vid. Selsk. Skr. I. Mat. Nat. Kl. Christiania}, 7:1--22,
  1906.

\bibitem[{Wood(2005)}]{Wood-DMTCS05}
\textsc{David~R. Wood}.
\newblock Acyclic, star and oriented colourings of graph subdivisions.
\newblock \emph{Discrete Math. Theor. Comput. Sci.}, 7(1):37--50, 2005.
\newblock
  \urlprefix\url{http://www.dmtcs.org/dmtcs-ojs/index.php/dmtcs/article/view/60}.
\newblock \msn{2164057}.

\end{thebibliography}

\small\def\soft#1{\leavevmode\setbox0=\hbox{h}\dimen7=\ht0\advance \dimen7
  by-1ex\relax\if t#1\relax\rlap{\raise.6\dimen7
  \hbox{\kern.3ex\char'47}}#1\relax\else\if T#1\relax
  \rlap{\raise.5\dimen7\hbox{\kern1.3ex\char'47}}#1\relax \else\if
  d#1\relax\rlap{\raise.5\dimen7\hbox{\kern.9ex \char'47}}#1\relax\else\if
  D#1\relax\rlap{\raise.5\dimen7 \hbox{\kern1.4ex\char'47}}#1\relax\else\if
  l#1\relax \rlap{\raise.5\dimen7\hbox{\kern.4ex\char'47}}#1\relax \else\if
  L#1\relax\rlap{\raise.5\dimen7\hbox{\kern.7ex
  \char'47}}#1\relax\else\message{accent \string\soft \space #1 not
  defined!}#1\relax\fi\fi\fi\fi\fi\fi}

\appendix

\section{Subdivisions via the Lov{\'a}sz Local Lemma}
\label{subLLL}

For the sake of comparing related proof techniques, we now give a
proof of a qualitatively similar result to
Theorem~\ref{thm:Subdivision} that uses the Lov{\'a}sz Local Lemma
instead of entropy compression. In particular, we use the following
`weighted' version of the Lov{\'a}sz Local Lemma; see \citep{MR02}.

\begin{lemma}
  \label{LLL}
  Let $\mathcal{E}=\{A_1,\dots,A_n\}$ be a set of ``bad'' events such
  that each $A_i$ is mutually independent of
  $\mathcal{E}\setminus(\mathcal{D}_i\cup\{A_i\})$ for some
  $\mathcal{D}_i\subseteq\mathcal{E}$. Let $p$ be a real number such
  that $0<p\leq\frac{1}{4}$. Let $t_1,\dots,t_n\geq1$ be integers
  called weights, such that for all $i\in[1,n]$,
  \begin{enumerate}
  \item[\textup{(a)}] $\text{\bf Pr}(A_i)\leq p^{t_i}$, and
  \item[\textup{(b)}]
    $\displaystyle2\sum_{A_j\in\mathcal{D}_i}(2p)^{t_j}\leq t_i
    \enspace.$
  \end{enumerate}
  Then with positive probability, no event in $\mathcal{E}$ occurs.
\end{lemma}

\begin{theorem}
  \label{OldSubdivision}
  For every graph $H$ with maximum degree $\Delta$, every
  subdivision $G$ of $H$ with at least $3+400\log\Delta$ division
  vertices per edge is nonrepetitively $23$-choosable.
\end{theorem}

\begin{proof}
  We may assume that $\Delta\geq 2$.  Let
  $r:=3+\ceil{400\log\Delta}$.  Let $G$ be a subdivision of $H$
  with at least $r$ division vertices per edge. Arbitrarily colour
  each original vertex of $H$ from its list. For each edge $vw$ of
  $H$, delete the colours chosen for $v$ and $w$ from the list of each
  division vertex on the edge $vw$. Now each division vertex has a
  list of at least $21$ colours. Colour each division vertex
  independently and randomly from its list. Let $p:=\frac{1}{21}$.

  Suppose that some path $P$ containing exactly one original vertex
  $v$ is repetitively coloured. Let $x$ be the vertex corresponding to
  $v$ in the other half of $P$. Thus $x$ is a division vertex of some
  edge incident to $v$, which is a contradiction since the colour of
  $v$ was removed from the list of colours at $x$. Thus no path with
  exactly one original vertex is repetitively coloured. Say an even
  path with no original vertices is \DEF{short}, and an even path
  with at least two original vertices is \DEF{long}. To prove that a
  colouring of $G$ is nonrepetitive it suffices to prove that no long
  or short path is repetitively coloured.

  Let $\mathcal{P}$ be the set of all short or long paths in $G$. Say
  $\mathcal{P}=\{P_1,\dots,P_n\}$ and each $P_i$ has $2\ell_i$
  vertices, of which $k_i$ are original vertices. Note that
  \begin{equation}
    \label{Long}
    2\ell_i\geq (k_i-1)(r+1)+1=(r+1)k_i-r\enspace.
  \end{equation}
  Orient each path $P_i$ so that the $j$-th vertex in $P_i$ is well
  defined. (Edges may be oriented differently in different paths.)\
  Let $A_i$ be the event that $P_i$ is repetitively coloured. Let
  $\mathcal{E}:=\{A_1,\dots,A_n\}$. If $P_i=(v_1,\dots,v_{2\ell_i})$
  then $v_j$ and $v_{\ell_i+j}$ are both division vertices for at
  least $\ell_i-k_i$ values of $j\in[1,\ell_i]$. Let $t_i:=\ell_i-k_i$
  be the weight of $A_i$ and of $P_i$.  Hence
  $\text{Pr}(A_i)\leq(\frac{1}{21})^{t_i}$, and Lemma~\ref{LLL}(a) is
  satisfied.

  We claim that $100 t_i\geq 99\ell_i$ for each path $P_i$.  This is
  immediately true if $k_i=0$. Now assume that $k_i\geq 2$. By
  \eqref{Long}, we have $2\ell_i \geq (r+1)(k_i-1) \geq 400(k_i-1)$
  and $\ell_i\geq 200$.  Thus $400 k_i\leq 2\ell_i+400$ and $400
  t_i=400\ell_i-400k_i\geq 398\ell_i-400\geq 396\ell_i$, as claimed.

  For each $i\in[1,n]$, let $\mathcal{D}_i$ be the set of events $A_j$
  such that the corresponding path $P_j$ and $P_i$ have a division
  vertex in common. Since division vertices are coloured
  independently, $A_i$ is mutually independent of
  $\mathcal{E}\setminus(\mathcal{D}_i\cup\{A_i\})$.

  Let $P_i\in\mathcal{P}$. Our goal is to bound the number of paths in
  $\mathcal{P}$ of given weight $t$ that intersect $P_i$.

  First consider the case of short paths of weight $t$. Such paths
  have order $2t$. Each vertex is in at most $2t$ short paths of order
  $2t$. Thus each vertex is in at most $2t$ short paths of weight
  $t$. Thus $P_i$, which has order $2\ell_i$, intersects at most
  $2\ell_i\cdot 2t\leq \frac{400}{99}t\,t_i$ short paths of weight
  $t$.

  Now consider the case of long paths with weight $t$. Let $P_j$ be
  such a long path. By \eqref{Long}, we have $(r+1)k_j\leq
  2\ell_j+r\leq 4\ell_j=4t+4k_j$.  Thus $k_j\leq \frac{4t}{r-3}$.
  Thus for each $q$, each vertex is the $q$-th vertex in at most
  $\Delta^{4t/(r-3)}$ long paths of weight $t$. Since $r-3\geq 400\log
  \Delta$, each vertex is the $q$-th vertex in at most $2^{t/100}$
  long paths of weight $t$.  Now $|P_i|=2\ell_i\leq \frac{200}{99}
  t_i$. Similarly, each path of weight $t$ has at most
  $\frac{200}{99}t$ vertices. Hence $P_i$ intersects at most
  $(\frac{200}{99})^2t_i t\, 2^{t/100}$ long paths of weight $t$. The
  same bound holds for short paths of weight $t$ (since
  $\frac{400}{99}t\,t_i\leq (\frac{200}{99})^2t_i t\, 2^{t/100}$).

  Thus Lemma~\ref{LLL}(b) is satisfied if for all $i$,
$$2\sum_{t\geq1} (\tfrac{200}{99})^2 t_i t \,2^{t/100}(\tfrac{2}{21})^{t}\leq t_i\enspace;$$
that is,
$$2 (\tfrac{200}{99})^2
\sum_{t\geq1} t \,(\tfrac{2}{21} \,2^{1/100})^{t}\leq 1\enspace.$$ For
$0<c<1$, we have $\sum_{t\geq 1}tc^t=\frac{c}{(1-c)^2}$.  Let
$c:=\tfrac{2}{21}\, 2^{1/100}\approx0.0959$.  Thus
$$2 (\tfrac{200}{99})^2 \sum_{t\geq1}  t
\,c^t=2 (\tfrac{200}{99})^2 \tfrac{c}{(1-c)^2}<0.958\enspace,$$ as
desired. Hence Lemma~\ref{LLL}(b) is satisfied.

Therefore with positive probability, no event in $\mathcal{E}$ occurs.
Thus, there exists a choice of colours for the division vertices such
that no event in $\mathcal{E}$ occurs.  That is, no short or long path
is repetitively coloured. Hence $G$ is nonrepetitively colourable from
the given lists.  Therefore $G$ is nonrepetitively $23$-choosable.
\end{proof}

\section{Caterpillars}
\label{Caterpillars}

Here we prove that every caterpillar is nonrepetitively
$148$-choosable. (Note that the constant $148$ can be significantly improved using
entropy-compression.)\ For $p\geq 2$, a coloured path $(v_1,\dots,v_{2p})$ of even order at least $4$ is 
\DEF{almost repetitively coloured} if $\psi(v_i)=\psi(v_{p+i+1})$ for all $i\in[1,p-1]$.
For $p\geq 1$, a coloured path $(v_1,\dots,v_{2p+1})$ of odd order at least $3$ is 
\DEF{almost repetitively coloured} if $\psi(v_i)=\psi(v_{p+i+1})$ for all $i\in[1,p]$. 

\begin{lemma}
  \label{Almost}
  Every path $G$ is $148$-choosable such that no path is repetitively
  coloured and no path is almost repetitively coloured.
\end{lemma}

\begin{proof}
  Colour each vertex $v$ of $G$ independently and randomly by a colour
  in the list of $v$. Let $P$ be a subpath of $G$ of order at least 2.
  If $|P|$ is even, then let $A_P$ be the event that $P$ is repetitively coloured; say $A_P$ has weight  $w(A_P):=\half |P|$. 
  If $|P|\geq 3$, then let $B_P$ be the event that $P$ is almost repetitively coloured. 
  If $|P|$ is even, then say $B_P$ has weight  $w(B_P):=\half |P|-2$.  
  If $|P|$ is odd, then say $B_P$ has weight  $w(B_P):=\half(|P|-1)$.  
    A vertex $v$ is said to be \emph{in} an event $E$ and $E$ \emph{contains} $v$ if $v$ is in the path  corresponding to $E$. 
      Let $\mathcal{E}$ be the set of all events. 
      
Observe that if an event $E\in\mathcal{E}$ has weight $t$ then   $\textbf{Pr}(E)\leq 148^{-t}$. 
Hence  Lemma~\ref{LLL}(a) is satisfied with $p:=148^{-1}$.  
For each event $E$ corresponding to some some subpath $P$, let  $\mathcal{D}(E)$  be the set of events that  contain a vertex in $P$.
  Thus  $E$ is mutually independent of  $\mathcal{E}\setminus\mathcal{D}(E)$.   Each vertex  is in at most $\ell$ subpaths of order $\ell$. 
  Each event of weight $t$ corresponds to a path of order $2t$ or $2t+1$ or $2t+2$.  
  So each vertex is in at most $6t+3$ events of weight $t$.  
Thus, if an event $E\in\mathcal{E}$ has weight $s$, then the path corresponding to $E$ has at most $2s+2$ vertices, implying $\mathcal{D}(E)$ 
contains  at most $(2s+2)(6t+3)\leq   36 s t$ events of weight $t$ (since $s,t\geq 1$).  
  
    To apply Lemma~\ref{LLL} we need
$$2\sum_{t\geq1}36st\,74^{-t}\leq s\enspace.$$
That is,
$$72\sum_{t\geq1}t \,74^{-t}\leq 1\enspace.$$
Now $\sum_{t\geq 1}tc^{-t}=\frac{c}{(c-1)^2}$ for all $c>1$.  Thus
$$72\sum_{t\geq1}t\,74^{-t}=\frac{72\cdot
  74}{73^2}<1\enspace,$$ as desired. Hence Lemma~\ref{LLL}(b) is
satisfied. Therefore with positive probability, no event in $\mathcal{E}$ occurs.
Thus, there exists a choice of colours such that no event in $\mathcal{E}$ occurs, in
which case no path is repetitively coloured and no path is almost
repetitively coloured.
\end{proof}


\begin{theorem}
  Every caterpillar is nonrepetitively $148$-choosable.
\end{theorem}

\begin{proof}
  Let $P$ be the spine of a caterpillar $T$.  By Lemma~\ref{Almost}, $P$ is $148$-choosable such that no
  subpath is repetitively coloured and no subpath is almost
  repetitively coloured.  Colour each leaf $x$ of $T$ by by an arbitrary
  colour in the list of $x$ that is distinct from the colour assigned
  to the $w$ neighbour of $v$ (which is in $P$) and is distinct from the colours assigned to the two neighbours of $w$ in $P$. 
Suppose, on the contrary, that there is a repetitively coloured path $Q$ in $T$.  
Since $P$ is nonrepetitively coloured,  $Q$ has at least one endpoint  that is a leaf in $T$. 
If exactly one endpoint of $Q$ is a leaf, then $Q\cap P$ is an almost repetitively coloured path of odd order, which is a contradiction.
Otherwise both endpoints of $Q$ are leaves. By the choice of colours for the leaves of $T$, we have $|Q|\geq 6$. 
Thus $Q\cap P$ is an almost repetitively coloured path of even order at least $4$, which is a contradiction.
This contradiction proves that there is no repetitively coloured path in $T$.  Therefore
$T$ is $148$-choosable.
 \end{proof}

\end{document}